\documentclass[reqno]{amsart}
\usepackage{xcolor}
\usepackage{amsthm,amsmath,amssymb,latexsym,soul,cite,mathrsfs}
\usepackage{enumitem}
\usepackage{color,enumitem}
\usepackage[colorlinks=true,urlcolor=blue,
citecolor=red,linkcolor=blue,linktocpage,pdfpagelabels,
bookmarksnumbered,bookmarksopen]{hyperref}
\usepackage[english]{babel}

\usepackage{graphicx}

\usepackage[left=2.7cm,right=2.7cm,top=2.9cm,bottom=2.9cm]{geometry}

\usepackage[hyperpageref]{backref}

\numberwithin{equation}{section}

\newtheorem{theorem}{Theorem}[section]
\theoremstyle{plain}
\newtheorem{theoremletter}{Theorem}

\newtheorem{lemma}[theorem]{Lemma}
\newtheorem{lemmaletter}[theoremletter]{Lemma}

\newtheorem{proposition}[theorem]{Proposition}

\theoremstyle{definition}
\newtheorem{remark}{Remark}[section]

\newcommand{\dx}{\,\mathrm{d}x}

\newcommand{\dtau}{\,\mathrm{d}\tau}

\newcommand{\dmu}{\,\mathrm{d}\mu}

\newcommand{\dive}{\mathrm{div}}

\DeclareMathOperator{\supp}{supp}
\DeclareMathOperator{\dist}{dist}

\newcommand{\loca}{\operatorname{loc}}

\usepackage[norefs,nocites]{refcheck}

\title[Nonlinear Quasilinear Schr\"{o}dinger Equation]{Existence of bound states for quasilinear elliptic problems involving critical growth and frequency}

\author[D.~Ferraz]{Diego Ferraz}
\address{Department of Mathematics,
	Federal University of Rio Grande do Norte
	59078-970, Natal-RN, Brazil}
\email{diego.ferraz.br@gmail.com}

\thanks{Corresponding author: Diego Ferraz}

\date{\today}

\begin{document}
	
\begin{abstract}
In this paper we study the existence of bound states of the following class of quasilinear problems,
\begin{equation*}
\left\{
\begin{aligned}
    &-\varepsilon ^p\Delta_pu+V(x)u^{p-1}=f(u)+u^{p^\ast -1},\ u>0,\   \text{in}\ \mathbb{R}^{N},\\
    &\lim _{|x|\rightarrow \infty }u(x) = 0 ,
\end{aligned}
\right.
\end{equation*}
where $\varepsilon>0$ is small, $1<p<N,$ $f$ is a nonlinearity with general subcritical growth in the Sobolev sense, $p^{\ast } = pN/(N-p)$ and $V$ is a continuous nonnegative potential. By introducing a new set of hypotheses, our analysis includes the critical frequency case which allows the potential $V$ to not be necessarily bounded below away from zero. We also study the regularity and behavior of positive solutions as $|x|\rightarrow \infty$ or $\varepsilon \rightarrow 0,$ proving that they are uniformly bounded and concentrate around suitable points of $\mathbb{R}^N,$ that may include local minima of $V$.
\end{abstract}
	
	\maketitle
	
	\section{Introduction}
In this paper we study the existence of positive bound states for a quasilinear elliptic problem of the type
\begin{equation}\tag{$P_\varepsilon$}\label{P}
\left\{
\begin{aligned}
    &-\varepsilon ^p\Delta_pu+V(x)u^{p-1}=f(u)+u^{p^\ast -1},\ u>0,\  \text{in}\ \mathbb{R}^{N},\\
    &\lim _{|x|\rightarrow \infty }u(x) = 0 ,
\end{aligned}
\right.
\end{equation}
where $\Delta  _pu = \dive (|\nabla u |^{p-2}\nabla u)$ is the $p$-Laplacian operator, $1<p<N,$ $0<\varepsilon <1,$ $f$ is a nonlinearity with general subcritical growth in Sobolev sense and $p^{\ast } = pN/(N-p)$ is the critical Sobolev exponent. Here the potential $V$ is a locally H\"{o}lder continuous function that may vanish in some regions of $\mathbb{R}^N$ with suitable conditions to be described as follows:
	\begin{enumerate}[label=($V_1$),ref=$(V_1)$]
		\item \label{V_autovalor}
		$V(x) \geq 0$ for all $x \in \mathbb{R}^N$ and
		\begin{equation*}
		d _1 := \inf \left\lbrace    \int _{\mathbb{R}^N} |\nabla u| ^p\mathrm{d}x +V(x)|u|^p\dx  :  u\in C_0 ^\infty  (\mathbb{R}^N)\text{ and } \int _{\mathbb{R}^N } |u| ^p\dx =1 \right\rbrace >0.
		\end{equation*}
	\end{enumerate}
	\begin{enumerate}[label=($V_2$),ref=$(V_2)$]
		\item\label{V_brutal} The set $\mathcal{O} = \{ x \in \mathbb{R}^N :  \inf_{\xi  \in \mathbb{R}^N} V(\xi ) =  V(x) = 0 \}$ is bounded.
	\end{enumerate}
        \begin{enumerate}[label=($V_3$),ref=$(V_3)$]
			\item\label{V_delpinofelmer}There is an open bounded smooth domain $\Omega \subset \mathbb{R}^N$ such that $\Omega  \supset \overline{\mathcal{O}}$ and 
			\begin{equation*}
			0 < V(x_0) < \inf_{\partial \Omega } V.
			\end{equation*}
	\end{enumerate}
With respect to the locally H\"{o}lder continuous function $f$, since we are looking for positive solutions, we consider $f(t) = 0$ when $t \leq 0,$ and $f(t) \geq 0,$ whenever $t\geq 0.$ We also take into account the next hypotheses:
	\begin{enumerate}[label=($f_1$),ref=$(f_1)$]
		\item \label{f_brutal} There are $p\leq q_0<q<p^\ast$ that satisfy: for any $\epsilon_0 >0$ there is $C_0 = C(\epsilon_0 )>0$ such that $f(t)  \leq \epsilon_0 t^{q_0-1}+ C_0 t^{q-1},$ for all $t \geq 0;$ Moreover, $q_0 =p,$ when $\inf _{\xi \in \mathbb{R}^N} V(\xi)>0,$ and $q_0 > p,$ if $\inf _{\xi \in \mathbb{R}^N} V(\xi)=0.$
	\end{enumerate}
	\begin{enumerate}[label=($f_2$),ref=$(f_2)$]
		\item \label{f_AR} There exists $p<\mu \leq p^\ast$ such that $\mu F(t) :=\mu  \int _0 ^t f (\tau ) \dtau \leq t f(t),$ for all $t \geq 0.$
	\end{enumerate}
	\begin{enumerate}[label=($f_3$),ref=$(f_3)$]
		\item \label{f_porbaixo} There are $\lambda>0$ and $p<p_0< p^\ast$ such that $F(t)  \geq \lambda t^{p_0},$ for all $t \geq 0.$
	\end{enumerate}
        \begin{enumerate}[label=($f_4$),ref=$(f_4)$]
			\item\label{f_increasing}The function $ t \mapsto (f(t) + t^{p^\ast -1} )t^{1-p}$ is increasing in $(0,\infty);$
	\end{enumerate}

	\subsection{Statement of the main result}
	Under \ref{V_autovalor} and \ref{f_brutal}, by a positive bound state solution (bound states) $u$ of the equation in \eqref{P}, we mean $u \in W^{1,p}(\mathbb{R}^N) ,$ where $W^{1,p}(\mathbb{R}^N)$ is the standard Sobolev space, $u>0$ in $\mathbb{R}^N$ and
	\begin{equation*}
		\varepsilon^p \int _{\mathbb{R}^N}|\nabla u | ^{p-2}\nabla u \cdot \nabla \varphi \dx +\int _{\mathbb{R}^N} V(x)u ^{p-1}\varphi \dx = \int _{\mathbb{R}^N}( f(u) + u^{p ^\ast-1 }) \varphi\dx,\ \forall \, \varphi \in C_0 ^\infty (\mathbb{R}^N).
	\end{equation*}
Our main result is concerned with the existence and the concentration behavior of positive bound states of \eqref{P} which concentrate around a possible local minima of $V.$
\begin{theorem}\label{teo_doo}
In addition to hypotheses \ref{V_autovalor}--\ref{V_delpinofelmer}, \ref{f_brutal}--\ref{f_porbaixo}, assume one of the following conditions:
	\begin{enumerate}[label=($p^\ast _i$),ref=$(p^\ast _i)$]
		\item\label{pi} $N \geq p ^2;$
	\end{enumerate}
	\begin{enumerate}[label=($p^\ast _{ii}$),ref=$(p^\ast _{ii})$]
		\item\label{pii} $p<N <p^2$ and $p^\ast - p/(p+1) - 1 < p_0 <p^\ast -1;$
	\end{enumerate}
	\begin{enumerate}[label=($p^\ast _{iii}$),ref=$(p^\ast _{iii})$]
		\item\label{piii} $p < N <p^2,$ $p ^\ast - p/(p+1) -1 \geq p_0 $ and $\lambda $ large enough;
	\end{enumerate}
Then, there exists $0<\varepsilon_0<1$ such that, for any $0 < \varepsilon <\varepsilon_0:$
\begin{enumerate}[label=(\roman*)]
    \item Problem \eqref{P} has a positive bound state solution $u_\varepsilon\in W^{1,p}(\mathbb{R}^N)\cap C^{1,\alpha }_{\loca}(\mathbb{R}^N);$
    \item $u_\varepsilon$ possesses a global maximum point $z_\varepsilon \in \Omega$ such that
    \begin{equation*}
		\lim _{\varepsilon \rightarrow 0} V(z _\varepsilon ) = V(x_0),
    \end{equation*}
    and $(u_\varepsilon(z_\varepsilon) )_{0<\varepsilon<\varepsilon_0}$ is bounded;
    \item There are $C,$ and $\alpha >0,$ in a such way that
		\begin{equation*}
		u_\varepsilon (x) \leq C \exp \left(  - \alpha \frac{|x - z_\varepsilon|}{\varepsilon}    \right),\ \forall \,x \in \mathbb{R}^N.
		\end{equation*}
\end{enumerate}
Furthermore, taking $\varepsilon \rightarrow 0,$ the functions $w_\varepsilon (x) = u_\varepsilon(\varepsilon x + z_\varepsilon)$ converge uniformly in compact sets of $\mathbb{R}^N$ and weakly in $D^{1,p}(\mathbb{R}^N)$ to a ground state (least energy among all other nontrivial solutions) $w_0\in W^{1,p}(\mathbb{R}^N)\cap C^{1,\alpha }_{\loca}(\mathbb{R}^N) $ of the problem
    \begin{equation}\label{PGS}
\left\{
\begin{aligned}
    &-\Delta_pw+V(x_0)w^{p-1}=f(w)+w^{p^\ast -1},\ w>0,\  \text{in} \ \mathbb{R}^{N},\\
    &\lim _{|x|\rightarrow \infty }w(x) = 0.
\end{aligned}
\right.
\end{equation}

\end{theorem}

\subsection{Motivation and related results}The study of \eqref{P} is mainly motivated by the semilinear case ($p=2$) and by looking for standing waves solutions of the nonlinear Schr\"{o}dinger equation
	\begin{equation}\label{PGERAL}
		i \hbar \frac{\partial \psi }{\partial t} = -\hbar ^2 \Delta_x \psi + V_0(x)  \psi  -  g_0 (|\psi |)\ \text{in}\ \mathbb{R}^N ,
	\end{equation}
where $i$ is the imaginary unit, $\hbar$ is the Planck constant and $g_0$ is a general continuous nonlinearity, such that $g_0 (\exp(-i \hbar^{-1} Et ) u ) = \exp(-i \hbar^{-1} E t )g_0(u).$ In fact, if $u$ is a solution of 
\begin{equation}\label{semilinear}
	\left\{
	\begin{aligned}
		&-\varepsilon ^2\Delta u+(V_0(x)-E)u=g(u) &\quad  \text{in} \ \mathbb{R}^{N},&\\
		&u\in  W^{1,2}(\mathbb{R}^N),\ u >0 &\quad  \text{in}\ \mathbb{R}^{N},&
	\end{aligned}
	\right.
\end{equation}
with $\varepsilon = \hbar$, then it is clearly that $\psi(x,t) =\exp(-i \hbar^{-1}Et )u(x)$ solves \eqref{PGERAL} with $g=g_0.$ When $\inf_{x \in \mathbb{R}^N} (V_0(x)-E)=0 ,$ the potential $V=V_0-E$ is said to have \textit{critical frequency} (see \cite{byeon-wang2002,byeon-wang2003,zhang2017}). For more details about \eqref{PGERAL} we refer to \cite{strauss1979,wang1993} and the references given there.

The study of \eqref{P} initially started by the semilinear case \eqref{semilinear} with nonlinearities $g$ having subcritical growth in Sobolev sense. In \cite{rabinowitz1992} P.H. Rabinowitz considered the condition,
\begin{equation}\label{rabinowitz}
\liminf _{|x| \rightarrow \infty}V(x) > \inf_{x\in \mathbb{R}^N} V(x)>0,
\end{equation}
in order to study existence of positive solutions for the equation in \eqref{semilinear} with $g(u)=u^{q},$ $1<q<(N+2)/(N-2)$ and small $\varepsilon,$ by using variational methods and applying suitable variants of the mountain pass theorem. In \cite{wang1993}, X. Wang proved, among several results, that positive ground state solutions of \eqref{semilinear} must concentrate at global minimum points of $V$, considering \eqref{rabinowitz} and $g(u)=u^{q}.$ Later on \cite{delpinofelmer}, M. del Pino and P. L. Felmer introduced the so called \textit{penalization method}, where they developed an approach to prove a local version of the results in \cite{rabinowitz1992}, taking into account the following different condition of \eqref{rabinowitz},
\begin{equation}\label{pot_delfelmer}
\inf _{x \in \Lambda }V(x) < \min _{x \in \partial \Lambda} V(x),
\end{equation}
for some suitable smooth bounded domain $\Gamma$. Roughly speaking the approach of \cite{delpinofelmer} involves the study of an auxiliary problem which consists of a new nonlinearity made by a modification on $g,$ where solutions of this auxiliary problem leads to solution of the initial one, provided $\varepsilon$ is small enough. Particularly, the method of \cite{delpinofelmer} allows to overcome the lack of compactness obtained by not taking \eqref{rabinowitz}. Moreover, considering a general nonlinearity $g$ with subcritical growth, the results of \cite{delpinofelmer} also showed that solutions of \eqref{semilinear} concentrate around local minimum points of $V$ when $\varepsilon$ goes to zero, extending and complementing the results of \cite{wang1993}. In \cite{alves-doo-souto2001} C. O. Alves et al. complemented the results of \cite{delpinofelmer} by considering \eqref{semilinear} and \eqref{pot_delfelmer} with a new nonlinear term involving critical Sobolev growth, more precisely, $g(u) = f(u)+u^{2^\ast-1}.$ For more results about the semilinear case with general nonlinearites we mention \cite{zhang-doo-minbo2017,jeanjean-tanaka2004,byeon-jeanjean2007}, and the references given there.

For the general case $1<p<N,$ the study of \eqref{P} have received an amount of considerable attention recently considering different classes of hypotheses on $V$ and $f$ (see \cite{doo2005,lyberopoulos2018,wang-an-zhang2015,figueiredo-furtado2012}). Next we describe some close related works with respect to our main result and on the arguments used to prove it. In \cite{lyberopoulos2018}, A. N. Lyberopouloss, using a similar approach with respect to our work, studied the existence of bound states for the following equation
\begin{equation*}
-\Delta _ p u + V(x)|u|^{p-2}u = K(x)|u|^{p-2}u + |u|^{p^\ast -2 }u,\quad x \in \mathbb{R}^N,
\end{equation*}
where $V$ and $K$ are nonnegative continuous potentials that may decay to zero as $|x|\rightarrow \infty$ but are free from any integrability or symmetry assumptions. Assuming \eqref{pot_delfelmer}, in \cite{figueiredo-furtado2012} G. M. Figueiredo and M. F. Furtado studied a multiplicity result for \eqref{P} by applying Lusternik-Schnirelmann theory, where they relate the number of positive solutions with the topology of the set where $V$ attains its minimum. On the other hand, supposing \eqref{rabinowitz}, J. Wang et al. in \cite{wang-an-zhang2015} proved that for $\varepsilon$ sufficiently small the problem \eqref{P} has a positive ground state solution with some concentration phenomena, where they also investigated multiplicity of solutions, existence and nonexistence of ground states for \eqref{P}. In \cite{doo2005}, under \eqref{pot_delfelmer}, J. M. do \'{O} proved the existence of positive bound states of \eqref{P} which concentrate around a local minima of $V,$ extending the results of \cite{alves-doo-souto2001} for the quasilinear framework.

It is worth to mention that in all those works the study of concentration phenomena (\cite{doo2005,wang-an-zhang2015,figueiredo-furtado2012}) for \eqref{P} as $\varepsilon \rightarrow 0,$ it is assumed that $V$ is bounded below away from zero, that is, 
\begin{equation*}
\inf_{x \in \mathbb{R}^N}V(x)>0.	
\end{equation*}
Based on this, a natural question arises: Is it possible to obtain existence of solutions and its concentration phenomena for \eqref{P}, whenever $1<p<N,$ when $\inf_{x \in \mathbb{R}^N}V(x)=0$ (critical frequency) with a nonlinear involving critical Sobolev growth? To the best of our knowledge there is no work in the present literature addressing this question. Nevertheless for the semilinear case \eqref{semilinear} with general nonlinearities $g$ in the subcritical growth range, that question was first considered by J. Byeon and Z. Wang in \cite{byeon-wang2002,byeon-wang2003} where the authors introduced the following condition over $V,$ 
\begin{equation*}
	\liminf _{|x| \rightarrow \infty }V(x) > \inf _{x \in \mathbb{R}^N} V(x)=0,
\end{equation*}
and developed a new variational approach to treat that case. Taking into account $m:=\inf _{x \in \Lambda }V(x)\geq 0 $ in \eqref{pot_delfelmer} the authors also constructed localized bound states and studied the concentration of solutions around local minimum of $V,$ which in that case can be positive or zero. In \cite{byeon-wang2002,byeon-wang2003} the behavior at infinity of $g$ in \eqref{semilinear} depends on whether $m$ is positive or not. The critical frequency involving nonlinearities with critical growth was first studied by J. Zhang in \cite{zhang2017}, where the author considered $V$ to be either compactly supported or decay faster than $|x|^{-2}$ at infinity, with aid of a truncation method given in \cite{zhang-doo2015TM}, extending the results of \cite{byeon-wang2002,byeon-wang2003}.

In our work we consider the critical frequency case $\inf_{x \in \mathbb{R}^N}V(x)=0$ in \eqref{P}, by introducing a new set of hypotheses \ref{V_autovalor}--\ref{V_delpinofelmer}. Our analysis for $p=2$ includes Problem \eqref{semilinear}, when is not required any condition on $V$ as $|x|\rightarrow \infty.$ In the study of \eqref{semilinear} by X. Wang in \cite{wang1993}, a similar condition with respect to \ref{V_autovalor} appeared ($p=2$) and it guarantees that the bilinear form $\int _{\mathbb{R}^N } \nabla u \cdot \nabla v + V(x)uv\dx$ is positive (coercive) for $u,v \in C_0^\infty (\mathbb{R}^N).$ Thus, it seems for us that \ref{V_autovalor} is the most reasonable condition in order to the norm operator $u \mapsto \int_{\mathbb{R}^N} |\nabla u|^p + V(\varepsilon x) |u|^p \dx $ be well $C^1$-defined on $W^{1,p}(\mathbb{R}^N).$ Proposition \ref{p_norma} shows the relation between \ref{V_autovalor} and its use on the natural Sobolev space related to the energy functional of \eqref{P}. Additionally, in a careful review on the argument of \cite{doo2005,alves-doo-souto2001}, it appears that hypotheses \ref{V_autovalor}--\ref{V_delpinofelmer} are the natural ones to consider when looking for concentration phenomena in the critical frequency case of \cite{doo2005,alves-doo-souto2001}. Nevertheless, in the sense of items (i)--(iii) of Theorem \ref{teo_doo}, we prove that the solutions of \eqref{P} concentrate at any point $x_0$ satisfying \ref{V_delpinofelmer}, i.e., conditions over the behavior of $V$ in the gap between a neighbourhood of $x_0$ and $\partial \Omega $ are not necessary. Therefore in Theorem \ref{teo_doo} we complement the study made in \cite{byeon-wang2002,byeon-wang2003,zhang2017} by considering the quasilinear case and the results of \cite{alves-doo-souto2001,doo2005} by taking the critical frequency case.

The proof of Theorem \ref{teo_doo} is inspired in the penalization method \cite{delpinofelmer,doo2005,alvessouto2012,liyang2013}. We first analyze the related auxiliary problem (see Section \ref{s_preli}) in an independent way, by just assuming \ref{V_autovalor}, \ref{V_brutal} and \ref{f_brutal}--\ref{f_porbaixo}, where some new regularity and uniform bounds in $\varepsilon$ for \eqref{P} are proved (see Lemma \ref{l_breziskato}). From this, using some recent arguments \cite{elves2009,ambrosio2020,lyberopoulos2018}, our proof allow to improve the bounded below away from zero case \cite[Theorem 1.1]{doo2005}, by taking into account \ref{f_increasing} together with \ref{V_autovalor}--\ref{V_delpinofelmer}, instead of \cite[Hyphothesis ($f_4$)]{doo2005}. 

\subsection{Additional remarks}\label{r_growth} Before we proceed, some comments on our hypotheses are necessary.
\begin{description}
			\item[(i)] Growth condition \ref{f_brutal} implies:  For any $\epsilon_0 >0$ there is $C_0 = C(\epsilon_0 )>0$ such that 
			\begin{enumerate}[label=(\alph*)]
				\item $f(t)  \leq \epsilon_0 t^{p-1}+ C_0 t^{q-1};$
				\item $f(t)  \leq \epsilon_0 t^{p-1}+ C_0 t^{p^\ast-1};$
				\item $f(t) \leq \epsilon_0 t^{q_0 - 1} + C_0 t^{p ^\ast -1};$
			\end{enumerate}
   			\item[(ii)] Condition \ref{f_brutal} includes nonlinearities $f$ such that $\limsup _{t\rightarrow 0_+} f(t)t^{1-q_0}=0$ and $f(t) \leq C(1 + t ^{q-1}),$ for all $t \geq 0,$ with $p\leq q_0<q<p^\ast$ and $C>0;$
            \item[(iii)] Condition \ref{f_brutal} differs of \cite[Hyphothesis ($f_1$)]{doo2005} because it only allows nonlinearities $f$ with the following behavior near the origin: $\lim _{t\rightarrow 0_+} f(t)t^{1-q_0}=0,$ where $p<q_0<p^\ast,$ when $\inf _{\xi \in \mathbb{R}^N}V(\xi ) = 0.$ This condition appears to be a natural constraint needed to deal with the critical frequency. Similar ones appeared in \cite{lyberopoulos2018,byeon-wang2003};
            \item[(iv)] Hypotheses \ref{f_AR} and \ref{f_porbaixo} imply $\lim _{t \rightarrow +\infty} f(t) t^{1-p} = + \infty;$
            \item[(v)] $V$ may not be a bounded function;
            \item[(vi)] If $V$ is bounded below away from zero, then it is clear that $V$ satisfy \ref{V_autovalor} and \ref{V_brutal}. For a more general nonnegative potential satisfying \ref{V_autovalor}, see Section \ref{s_preli} -- Remark \ref{r_potential};
            \item[(vii)] Under our conditions, it is expected that Problem \eqref{P} does not admit ground states (see \cite{alves-souto2002});
            \item[(viii)] If $f$ is locally Lipschitz and $1<p\leq 2,$ one can follow the same argument made in \cite{doo2005} by applying the regularity results of \cite{damascelli-ramaswamy2001}, to conclude that the global maximum points $z_\varepsilon$ are unique;
\end{description}	
\subsection{Outline} In Section \ref{s_preli} we introduce the variational setting that we use, summing up the main properties of the penalization method. In Section \ref{s_existence} we prove existence, regularity and uniform bounds of solutions of the auxiliary problem with respect to \eqref{P}. Section \ref{s_teodoo} is dedicated to complete the proof of Theorem \ref{teo_doo}. If not stated, we always assume that \ref{V_autovalor}, \ref{f_brutal}, \ref{f_AR} and \ref{f_porbaixo} hold with $0 < \varepsilon < 1.$\\

\noindent \textbf{Notation:} In this paper, we use the following
	notations:
	\begin{itemize}
		\item  The usual norms in $L^{p}(\mathbb{R}^N)$ are denoted by $\|\cdot  \| _p;$
		\item $\| \cdot \| _{1,p}$ is the usual Sobolev norm of $W^{1,p}(\mathbb{R}^N);$
		\item $X_\varepsilon$ is the space $W^{1,p}(\mathbb{R}^N)$ equipped with the norm $\| \cdot  \|_{(\varepsilon)}$ (see Proposition \ref{p_norma});
		\item $B_R(x_0)$ is the $N$-ball of radius $R$ and center $x_0;$ $B_R:=B_R(0);$
		\item  $C_i$ denotes (possible different) any positive constant;
		\item $\mathcal{X}_A$ is the characteristic function (indicator function) of the set $A \subset \mathbb{R}^N;$
		\item $A^c =\mathbb{R}^N \setminus A,$ for $A \subset \mathbb{R}^N;$ 
		\item $u^+= \max\{u,0  \}$ and $u^- = \max\{-u,0  \};$
		\item $|A|$ is the Lebesgue measure of the mensurable set $A \subset \mathbb{R}^N;$
		\item $[|u|\leq \delta] = \{ x \in \mathbb{R}^N:u(x) \leq \delta \}.$ Similarly we denote $[|u|< \delta],$ $[|u|> \delta]$ and $[|u| \geq  \delta];$
	\end{itemize}

	
	\section{Variational setting: Del Pino and Felmer's approach}\label{s_preli}
	Following \cite{delpinofelmer}, we take into account that the change of variables $v(x) = u( \varepsilon x)$ in Problem \eqref{P} leads to the equivalent problem
\begin{equation}\label{P0}\tag{$\hat{P}_\varepsilon$}
\left\{
\begin{aligned}
    &-\Delta_pv+V(\varepsilon x)v^{p-1}=f(v)+v^{p^\ast -1},\ v>0,\  \text{in}\ \mathbb{R}^{N},\\
    &\lim _{|x|\rightarrow \infty }v(x) = 0 ,
\end{aligned}
\right.
\end{equation}
	in the sense that solutions of \eqref{P} give solutions of \eqref{P0} and reciprocally, by taking $u(x) = v(x/\varepsilon)$. Now let us fix a bounded domain $\Omega \subset \mathbb{R}^N$ with $\overline{\mathcal{O}}\subset \Omega $ and $0 \in \Omega.$ We denote
	\begin{equation*}
	V _0 := \inf \left\lbrace V(\varepsilon x) : x \not \in  \Omega _\varepsilon \right\rbrace = \inf \left\lbrace  V(y) : y \not \in \Omega \right\rbrace >0,
	\end{equation*}
	where $\Omega _\varepsilon = \varepsilon^{-1}\Omega.$ In particular,
	\begin{equation}\label{usa}
	V(\varepsilon x) \geq  V_0,\quad \text{for any}\ x \in \Omega _\varepsilon^c = \varepsilon^{-1}\Omega ^c.
	\end{equation}
	On the other hand, since
	\begin{equation*}
	\lim _{t \rightarrow 0} \frac{f(t) + t^{p^\ast -1 }}{t^{p-1}}=0\quad\text{and}\quad\lim _{t \rightarrow \infty }\frac{f(t) + t^{p^\ast -1}}{t^{p-1}} = + \infty,
	\end{equation*}
	one can define, 
	\begin{equation*}
	a_0 := \min\left\lbrace b >0: \frac{f( b ) + b ^{p^\ast -1} }{b ^{p-1}}= \frac{V_0}{\kappa_0}  \right\rbrace,
	\end{equation*}
	where $\kappa_0 >1.$ Furthermore, we consider a continuous function
	\begin{equation*}
	f_\ast  (t) = 
	\left\{  
	\begin{aligned}
	&0, &&\text{ if }t \leq 0,\\
	&f(t) + t^{p^\ast -1}, && \text{ if } 0 \leq t \leq a_0,\\
	&\frac{V_0}{\kappa_0} t ^{p-1}, &&\text{ if }t \geq a _0,
	\end{aligned}
	\right.
	\end{equation*}
	and the following auxiliary function is defined
	\begin{equation*}
	g(x,t) = 
	\left\{
	\begin{aligned}
	&\mathcal{X}_\Omega (x) ( f(t) + t^{p^\ast -1}) + (1 - \mathcal{X}_\Omega (x) ) f_\ast (t),&&\text{ if }t \geq 0,\ x \in \mathbb{R}^N .\\
	&0,&&\text{ if }t \leq 0,\ x \in \mathbb{R}^N .
	\end{aligned}
	\right.
	\end{equation*}
	\begin{proposition}\label{p_ge}
		The function $g$ is a Carath\'{e}odory function and satisfy:
		\begin{enumerate}[label=(\roman*)]
			\item There is $c_f>0$ such that $g(x,t) \leq  c_f\left( f(t) + t^{p^\ast -1} \right),$ for all $x \in \mathbb{R}^N$ and $t \geq 0$ ($c_f=1$ when $f$ satisfies \ref{f_increasing});
			\item $\displaystyle \mu G(x,t) := \mu \int _0 ^t g(x, \tau ) \dtau  \leq g(x,t) t,$ for all $x \in \Omega$ or $0 \leq t \leq a_0;$
			\item $g(x,t) \leq  (V _0 / \kappa _0)t^{p-1},$ for all $x \not \in \Omega $ and $t \geq 0;$
			\item $\mathcal{G}(x,t) := g(x,t)t -p G(x,t) \geq 0$ for all $x \in \mathbb{R}^N$ and $t \geq 0;$
			\item If \ref{f_increasing} holds then $t\mapsto g(x,t)/t^{p-1}$ is nondecreasing for all $x \in \mathbb{R}^N$ and $t > 0;$
		\end{enumerate}
	\end{proposition}
	The approach due to del Pino-Felmer \cite{delpinofelmer} consists on the study of the following \textit{auxiliary problem}
\begin{equation}\label{PF}\tag{$PF_\varepsilon$}
\left\{
\begin{aligned}
    &-\Delta_pv+V(\varepsilon x)v^{p-1}=g(\varepsilon x,v),\ v>0,\  \text{in}\ \mathbb{R}^{N},\\
    &\lim _{|x|\rightarrow \infty }v(x) = 0.
\end{aligned}
\right.
\end{equation}

	If $v_\varepsilon$ is a solution of \eqref{PF} such that $v_\varepsilon(x) \leq a_0,$ for all $x \not \in \Omega_\varepsilon,$ then $v_\varepsilon$ is a solution of \eqref{P0}. Following we define $X_\varepsilon$ as the completion of $C^\infty _0 (\mathbb{R}^N)$ with respect to the norm
	\begin{equation*}
	\| v \| _{(\varepsilon)} := \left[  \int _{\mathbb{R}^N} | \nabla v |^p  +    V(\varepsilon x) |v|^p \dx   \right]^{1/p},\quad v \in C^\infty _0 (\mathbb{R}^N).
	\end{equation*}
	$X_\varepsilon$ is the space of functions to look for solutions of \eqref{PF}.
	\begin{proposition}\label{p_norma}
		If $0 \leq  \varepsilon \leq 1$ and \ref{V_autovalor} hold, then $(X_\varepsilon,\| \, \cdot \,  \| _{(\varepsilon)} )$ is a well defined Banach space. Furthermore, $X_\varepsilon$ is continuoulsy embedded in $W^{1,p}(\mathbb{R}^N)$ and
		\begin{equation*}
		X_\varepsilon = \left\lbrace  v \in W^{1,p}(\mathbb{R}^N) : \int _{\mathbb{R}^N }V(\varepsilon x) |v|^p \dx <+ \infty   \right\rbrace.
		\end{equation*}
	\end{proposition}
	\begin{proof}
		We start the proof by observing that \ref{V_autovalor} implies
		\begin{equation*}
		d_\varepsilon := \inf \left\lbrace    \| v \| ^p _{(\varepsilon) }:  v\in C_0 ^\infty  (\mathbb{R}^N)\text{ and } \|v \|_p =\varepsilon ^{-N/p} \right\rbrace \geq d_1 >0,\quad\text{for }0 < \varepsilon \leq 1.
		\end{equation*}
		In fact, by using the change of variable $x \mapsto \varepsilon x,$ we see that
		\begin{align*}
		\int _{\mathbb{R}^N} | \nabla v |^p  +    V(\varepsilon x) |v|^p \dx &= \varepsilon ^{p - N} \int _{\mathbb{R}^N}  |\nabla v_\varepsilon | ^p  \dx + \varepsilon ^{-N} \int _{\mathbb{R}^N} V(x)  |v _\varepsilon |^p   \dx,\\
		\int _{\mathbb{R}^N}|v|^p \dx &=\varepsilon ^{-N}\int _{\mathbb{R}^N}  |v _\varepsilon |^p  \dx.
		\end{align*}
		where $v_\varepsilon (x)= v (x/\varepsilon)$ and $v \in C_0 ^\infty (\mathbb{R}^N).$ In particular, if $\| v \| _p = \varepsilon ^{-N/p},$ then $\| v _\varepsilon \| _p =1$ and $\| v \|_{(\varepsilon)}^p \geq \| v_\varepsilon \|_{(1)}^p \geq d_1,$ provided $ 0 < \varepsilon \leq 1.$ Therefore $d_\varepsilon \geq d_1.$ Now let $w_\varepsilon = v \| v \|_p^{-1} \varepsilon ^{-N/p},$ $v \in C_0 ^\infty(\mathbb{R}^N)\setminus \{ 0 \}.$ Then $\| w_\varepsilon \| _p  = \varepsilon ^{-N/p}$ and so
		\begin{equation*}
		\int _{\mathbb{R}^N } |v|^p\dx \leq \frac{1}{\varepsilon ^N d_\varepsilon } \int _{\mathbb{R}^N }|\nabla v|^p + V(\varepsilon x) | v | ^p \dx.
		\end{equation*}
		Consequently, 
		\begin{equation*}
		\| v \| _{1,p} \leq \left( 1+\frac{1}{\varepsilon ^N d_1} \right)^{1/p} \| v \|_{(\varepsilon)},\quad \forall \, v \in C_0 ^\infty(\mathbb{R}^N).
		\end{equation*}
		In particular, for any $(\varphi_k)_k \subset C^\infty _0 (\mathbb{R}^N),$ $\| \varphi _k - \varphi _l  \|_{(\varepsilon)} \rightarrow 0$ implies $\| \varphi _k - \varphi _l  \|_{1,p} \rightarrow 0$ as $k,l \rightarrow \infty.$ Hence $X_\varepsilon$ is a well defined space of functions and by an application of Fatou's lemma,
		\begin{equation}\label{inclusao_espaco}
		X_\varepsilon \subset \left\lbrace  v \in W^{1,p}(\mathbb{R}^N) : \int _{\mathbb{R}^N }V(\varepsilon x) |v|^p \dx < + \infty   \right\rbrace .	
		\end{equation}
		We have to prove  the converse inclusion of \eqref{inclusao_espaco}. Let $v \in W^{1,p} (\mathbb{R}^N)$ such that $\int_{\mathbb{R}^N} V(\varepsilon x)|v|^p\dx < \infty.$ Assume first that $v$ has compact support. Let $(\varrho_k)$ the standard sequence of mollifiers and define $v_k = \varrho_k \ast v \in C^\infty_0(\mathbb{R}^N).$ By Friedrichs theorem we have $v_k \rightarrow v$ and $\nabla v_k \rightarrow \nabla v$ in $L^p(\mathbb{R}^N).$ Since $K:=\supp(v) \cup \overline{B_1 }$ is a compact set,
		\begin{equation*}
		\int_{\mathbb{R}^N}V(\varepsilon x)|v_k -v|^p \dx\leq \|V(\varepsilon x)\|_{L^\infty(K)}\int_{\mathbb{R}^N} |v_k -v|^p\dx \rightarrow 0, \text{ when }k \rightarrow \infty.
		\end{equation*}
		Thus, in that case, $v \in X_\varepsilon.$ For the general case, we consider a truncation function $\xi \in C_0 ^\infty (\mathbb{R}^N)$ such that $\xi= 1$ in $B_1,$ $\xi= 0$ on $ B^c_2$ and $\| \nabla \xi\|_\infty \leq M.$ Now we define $\xi _k (x) = \xi (x/k)$ and take $v_k = \xi_k v \in W^{1,p} (\mathbb{R}^N),$ which has compact support. In particular, as we saw, $v_k \in X_\varepsilon.$ It is easy to see, using Lebesgue's dominated convergence theorem, that $\|v_k - v\|_{(\varepsilon)} \rightarrow 0.$ Since $X_\varepsilon$ is a closed space, we have $v\in X_\varepsilon.$
	\end{proof}
\begin{remark}\label{r_potential}
Consider $V-\delta_0 \in L^s([V<\delta _0]),$ $s>N/p$, where $\delta_0 >0$ is such that
\begin{equation*}
\| V(x) - \delta_0 \| _{L^s ([V<\delta_0 ])} < 2^{-1}(\min\{ 1,\delta_0  \})C_{p,s},
\end{equation*}
with
\begin{equation*}
0<C_{p,s} = \inf \left\lbrace  \| u \|^p _{1,p}\| u \|^{-p} _{ps'} : u \in C^\infty _0 (\mathbb{R}^N) \right\rbrace <+\infty,
\end{equation*}
and $s' = s/(s-1).$ Then $V$ satisfies \ref{V_autovalor}. In fact, for $u \in C^\infty _0(\mathbb{R}^N)$ such that $\| u \|_p =1,$ we have
\begin{equation*}
\| u \|^p_{(1)} \geq \min\{ 1,\delta_0  \} \| u \|^p_{1,p} +\int_{[V<\delta_0 ]} (V(x) - \delta_0) |u|^p \dx .
\end{equation*}
On the other hand, by H\"{o}lder's inequality,
\begin{equation*}
\int _{[V<\delta_0 ]} (\delta_0-V(x))|u|^p \dx \leq \| V(x) - \delta_0 \| _{L^s ([V<\delta_0 ])} \| u \|^p_{L^{ps'}([V<\delta_0 ])}\leq 2^{-1}\min\{ 1,\delta_0  \} \| u \|^p_{1,p}.
\end{equation*}
Consequently,
\begin{equation*}
\| u \|^p_{(1)} \geq 2^{-1}\min\{ 1,\delta_0  \} \| u \|^p_{1,p} \geq 2^{-1}\min\{ 1,\delta_0  \},
\end{equation*}
which leads to $d_1\geq 2^{-1}\min\{ 1,\delta_0  \} >0.$
\end{remark}
	Related to Problem \eqref{PF}, we consider the following functional $I_\varepsilon : X_\varepsilon \rightarrow \mathbb{R}$ given by
	\begin{equation*}
	I_\varepsilon (v) = \frac{1}{p} \| v \|_{(\varepsilon)}^p - \int _{\mathbb{R}^N } G(\varepsilon x,v)\dx.
	\end{equation*}
	It is easy to see that $I_\varepsilon \in C^1 (X_\varepsilon : \mathbb{R})$ and critical points of $I_\varepsilon$ are weak solution of \eqref{PF}, in the sense that
	\begin{equation*}
	\int _{\mathbb{R}^N}|\nabla v _\varepsilon| ^{p-2}\nabla v _\varepsilon\cdot \nabla \varphi\dx +\int _{\mathbb{R}^N} V( \varepsilon  x)|v_\varepsilon| ^{p-2}v_\varepsilon \varphi \dx = \int _{\mathbb{R}^N}g(\varepsilon x , v_\varepsilon) \varphi\dx,\ \forall \, \varphi \in W^{1,p} (\mathbb{R}^N).
	\end{equation*}
	\section{Existence and regularity of solutions for the auxiliary problem}\label{s_existence}
	In this section we prove that under our assumption $I_\varepsilon$ has a solution at the Mountain Pass level
	\begin{equation*}
	c(I_\varepsilon) := \inf _{\gamma \in \Gamma_{I \varepsilon}} \sup _{t \geq 0} I_\varepsilon(\gamma (t)),
	\end{equation*}
	where 
	\begin{equation}\label{gamma}
	\Gamma_{I_\varepsilon} = \left\{   \gamma \in C([0,\infty), X _\varepsilon): \gamma(0)=0 \mbox{ and } \lim_{t \rightarrow \infty} I_\varepsilon(\gamma (t)) = - \infty   \right\}.
	\end{equation}
	\begin{lemma}\label{l_geometry}
		$I_\varepsilon$ has the Mountain Pass geometry. More precisely:
		\begin{enumerate}[label=(\roman*)]
			\item There exist $r,b>0$ such that $I_\varepsilon(v) \geq b,$ whenever $\|v\|_{(\varepsilon)} = r; $
			\item There is $e \in X_\varepsilon$ with $\|e\| _{(\varepsilon)}> r$ and $I_\varepsilon(e)<0.$ 
		\end{enumerate}
		In addition $0<c(I_\varepsilon)\leq \max _{t \geq 0} I_\varepsilon (tv),$ for any $ v \in X_\varepsilon$ such that $\| v ^+\|_{L^{p ^\ast } (\Omega _\varepsilon) }>0.$ Moreover, if $v=v_\varepsilon$ is a nontrivial critical point of $I_\varepsilon$ and \ref{f_increasing} holds, then $\max _{t \geq 0}I_\varepsilon (t v_\varepsilon) = I_\varepsilon (v_\varepsilon).$
	\end{lemma}
	\begin{proof}
		Clearly \ref{f_brutal} and Sobolev inequalities imply
		\begin{equation*}
		I_\varepsilon(v) \geq \| v \| _{(\varepsilon)} ^p \left(  \frac{1}{p}- \epsilon_0C_1 \| v \|_{(\varepsilon)}^{q_0 - p} -C_2(\epsilon_0)\| v \| _{(\varepsilon)}^{q - p}- C_{3} \| v \| _{(\varepsilon)} ^{p^\ast - p }   \right)>0,
		\end{equation*}
		for suitable $\epsilon_0>0$ and $\| v \|_{(\varepsilon)} >0$ small enough. On the other hand, taking $ v \in X_\varepsilon$ such that $\| v^+ \|_{L^{p ^\ast } (\Omega _\varepsilon) }>0,$ we see that
		\begin{equation*}
		I_\varepsilon(t v) \leq \frac{t^p}{p} \| v \|^p _{(\varepsilon)}-\frac{t^{p^\ast}}{p^\ast}\int _{\Omega_\varepsilon} (  v^+ )^{p^\ast}\dx\rightarrow -\infty, \ \text{as }t \rightarrow +\infty.
		\end{equation*}
		To prove the last statement we consider $\varphi (t) = I_\varepsilon (t v_\varepsilon),$ $t \geq 0.$ Since $I_\varepsilon'(v_\varepsilon) = 0,$ the following identity holds
		\begin{equation*}
		\| v_\varepsilon \| _{(\varepsilon)}^p= \int _{\mathbb{R}^N}  g(\varepsilon x, v_\varepsilon) v_\varepsilon\dx.
		\end{equation*}
		Hence
		\begin{equation*}
		\varphi' (t) = \int _{\mathbb{R}^N} t^{p-1} v_\varepsilon^p \left( \frac{g(\varepsilon x , v_\varepsilon)}{v_\varepsilon^{p-1}} - \frac{g(\varepsilon x, t v_\varepsilon)}{(tv_\varepsilon )^{p-1}}  \right) \dx,\ \forall \, t > 0.
		\end{equation*}
		Now we use Proposition \ref{p_ge}--(v) to see that $\varphi' (t) \geq 0,$ whenever $0\leq t \leq 1,$ and $\varphi' (t) \leq 0,$ when $t \geq 1.$ Therefore $t=1$ is maximum point of $\varphi.$
	\end{proof}
	\begin{remark}
		We apply a variant of the mountain pass theorem given by \cite[Theorem 2.1]{elves2009}. Regarding that, following \cite{elves2009} we consider $\hat{c}(I_\varepsilon) = \inf_{\gamma \in \hat{\Gamma} _{I_\varepsilon}} \max_{t \in [0,1]} I_\varepsilon(\gamma (t)),$ where
		\begin{equation*}
		\hat{\Gamma}_{I_\varepsilon} = \left\{   \gamma \in C([0,1], X _\varepsilon): \gamma(0)=0,\ \| \gamma (1) \|_{(\varepsilon)} >r \mbox{ and } I_\varepsilon(\gamma (1)) <0 \right\}.
		\end{equation*}
		Then $\hat{c}(I_\varepsilon) = c(I_\varepsilon).$ In fact, for a given $\gamma \in \Gamma _{I_\varepsilon}$ there is $t _\gamma >0$ such that $I(\gamma (t)) <0$ for any $t\geq t_\gamma$ with $\| \gamma (t_\gamma)\|>r.$ Define $\hat{\gamma}(t) = \gamma (tt_\gamma ),$ for $0 \leq t \leq 1.$ We have $\hat{\gamma} \in \hat{\Gamma}_{I_\varepsilon},$ 
		\begin{equation*}
		\hat{c}(I_\varepsilon) \leq \max _{t \in [0,1]}I _\varepsilon(\hat{\gamma} (t)) = \max _{t  \geq 0}I _\varepsilon(\gamma(t)),
		\end{equation*}
		and so $\hat{c}(I_\varepsilon) \leq c(I_\varepsilon).$ On the other hand, let $\hat{\gamma} \in \hat{\Gamma}_I$ and take $\varphi \in X_\varepsilon$ with $\varphi \geq |\hat{\gamma}(1)|$ in $\Omega _\varepsilon.$ Let $t_0>1$ such that
		\begin{equation*}
		\frac{t^p}{p} \left\| \hat{\gamma}(1)/t_0^2 + (1/t_0 - t_0/t)\varphi      \right\|^p_{(\varepsilon)}  - \frac{t^{p^\ast}}{p^{\ast}} \int _{\Omega _\varepsilon }((\hat{\gamma}(1)/t_0^2 + (1/t_0 - t_0/t)\varphi    )^+)^{p^\ast} \dx <0,
		\end{equation*}
		for $t \geq t^2_0+1.$ We define
		\begin{equation*}
		\gamma (t) =
		\left\{  
		\begin{aligned}
		&\hat{\gamma }(t), &&\text{ if }0\leq t \leq 1,\\
		&\hat{\gamma }(1), && \text{ if } 1\leq t \leq t^2_0,\\
		&(t/t^2_0) \hat{\gamma }(1) +((t/t_0)-t_0) \varphi, && \text{ if } t \geq t^2_0.
		\end{aligned}
		\right.
		\end{equation*}
		Then $\gamma \in \Gamma _{I_\varepsilon},$ $c(I_\varepsilon) \leq \max _{t \geq 0}I _\varepsilon(\gamma (t)) = \max _{t  \geq 0}I _\varepsilon(\hat{\gamma}(t))$ and consequently $\hat{c}(I_\varepsilon) = c(I_\varepsilon).$
	\end{remark}
Next, we have a minimax estimate.
\begin{lemma}\label{l_minimax}
	If either \ref{pi}, \ref{pii} or \ref{piii} hold, then $c(I_\varepsilon) < ( 1/N )\mathbb{S}^{N/p},$ for all $\varepsilon >0,$ where
\begin{equation}\label{sobolev_constant}
\mathbb{S} = \inf  \left\lbrace   \|   \nabla  u \|_p ^p   : u \in D^{1,p}(\mathbb{R}^N ) \text{ and }\| u \| _{p^\ast }=1 \right\rbrace .
\end{equation}
\end{lemma}
\begin{proof}
	Consider $\varrho _0>0$ in such way that $B_{\varrho _0} \subset \Omega.$ Now take $\psi \in C_0 ^\infty(\mathbb{R}^N),$ with $0\leq \psi \leq 1 $ and such that $\psi(x) =1,$ if $|x| \leq  \varrho_0/2,$  and $\psi(x) =0,$ if $|x| \geq \varrho_0.$ For 
	\begin{equation*}
	W_\sigma (x) = (\sigma + |x|^{p/(p-1)}) ^{(p-N)/p}, \text{ denote } \nu  _\sigma := \psi W _\sigma \text{ and } \omega _\sigma =   \nu  _\sigma  / \|  \nu  _\sigma  \|_{p ^\ast}.
	\end{equation*}
	Using the fact that $\mathbb{S} $ is attained by the following functions \cite{talenti},
	\begin{equation*}
	U _\sigma (x) = C(N,p) \frac{\sigma ^{ (N-p)/ p^2 }}{   \left( \sigma + |x|^{p/(p-1)}  \right) ^{(N-p) / p}  }    ,\ \sigma >0, \text{ where }C(N,p) = \left[  N \left( \frac{N-p}{p-1}  \right) ^{p-1}   \right]^{(N-p) / p^2},
	\end{equation*}
	one can prove the following asymptotic behavior of $\omega_\sigma$ for small values of  $\sigma\in (0,1)$ (see \cite[Lemma 3.2]{guedda-veron1989} and \cite{lyberopoulos2018}):
	\begin{align}
	\| \nabla \omega  _\sigma \| _p ^p &= \mathbb{S} + \mathcal{O}(\sigma ^{(N-p) / p}),\label{ozao1}\\
	\|\omega  _\sigma  \| _s ^s
	&=\left\{
	\begin{array}{ll}
	\mathcal{O} (\sigma ^{N (1-s/p^\ast)(p-1)/ p}  ) , &\text{ if }s>p^\ast (1-p^{-1}),\\
	\mathcal{O}(\sigma ^{s (N-p) /p^2} |\ln \sigma|), &\text{ if }s = p^\ast (1-p^{-1}),\\
	\mathcal{O}(\sigma ^{s(N-p)/p^2 }), &\text{ if }s < p^\ast (1-p^{-1}),
	\end{array}
	\right.\\
	\|\omega  _\sigma \| _p ^p & \leq \left\{
	\begin{array}{ll}
	\mathcal{O} (\sigma ^{ p-1}  ) , &\text{ if }p^2<N,\\
	\mathcal{O} ( \sigma ^{ p-1} |\ln \sigma| ), &\text{ if }p^2=N,\\
	\mathcal{O} ( \sigma ^{ (N-p) / p} ), &\text{ if }p^2 > N,
	\end{array}
	\right.\label{ozao3}
	\end{align}
	where the Bachmann–Landau notation $\mathcal{O} (\alpha (\sigma)) = \beta (\sigma)$ stands for the existence of $c_1, c_2> 0$ such that $c_1 \leq \beta (\sigma)/\alpha (\sigma) \leq c_2.$ In particular $\lim _{\sigma \rightarrow 0 }\| \omega _\sigma \|_s^s = 0 $ for all $p\leq s<p^\ast.$ Moreover, $g(x,\omega _\sigma ) = f(\omega _\sigma) + \omega _\sigma ^{p ^ \ast -1}$ in $B_{\varrho _0}$ and so one can follow \cite[Lemma 6]{lyberopoulos2018}. In fact, the path $ t \mapsto t \omega _\sigma,$ $t \geq 0,$ belongs to $\Gamma _{I_\varepsilon}$ and by Lemma \ref{l_geometry} it suffices to prove that $\max _{t \geq 0} I_\varepsilon (t \omega _\sigma) <  ( 1/N )\mathbb{S}^{N/p}.$ In order do that, we notice that \ref{f_porbaixo} yields
	\begin{equation*}
	I_\varepsilon (t \omega _\sigma)  \leq z _\sigma (t) := \frac{t^p}{p} \| \omega _\sigma \| _{(\varepsilon)} ^p - \lambda t ^{p_0} \| \omega _\sigma \|^{p_0} _{L ^{p_0} (B _{\varrho_\varepsilon})} - \frac{t^{p ^\ast}}{p^\ast},\quad t \geq 0,
	\end{equation*}
	where $B_{\varrho _\varepsilon} = \varepsilon^{-1}B_{\varrho_0}.$ Clearly $z_\sigma $ attains a unique global maximum point $t_\sigma,$ which satisfies
	\begin{equation*}
	\lambda p_0\| \omega _\sigma \|^{p_0}_{L ^{p_0} (B _{\varrho_\varepsilon} )} t_{\sigma} ^{p_0 - p }+t_{\sigma } ^{p^{\ast }   - p } = \| \omega_ \sigma \| _{(\varepsilon)} ^p.
	\end{equation*}
	Using estimates \eqref{ozao1}--\eqref{ozao3} we have the existence of $b_1,$ $b_2>0$ such that $0<b_1\leq t_\sigma  \leq b_2,$ for $0<\sigma<1$ small enough. Now we notice that we can write
	\begin{align*}
	\max _{t \geq 0} I_{\varepsilon} (t \omega _\sigma )  & \leq \max _{t \geq 0} \overline{z} _\sigma (t) - \lambda b_1^{p_0} \| \omega _\sigma \|^{p_0}_{L ^{p_0} (B_{\varrho_\varepsilon})}\\
	& = \frac{1}{N} \| \omega _\sigma \| _{(\varepsilon)} ^N - \lambda b_1^{p_0} \| \omega _\sigma \|^{p_0}_{L ^{p_0} (B_{\varrho_\varepsilon})},
	\end{align*}
	where $\overline{z} _\sigma (t) = p^{-1}\| \omega _\sigma \|_{(\varepsilon)} ^p  t^p - (p^\ast)^{-1} t^{p^ \ast}$ has a unique maximum point $\overline{t}_\sigma  = \| \omega _\sigma  \|_{(\varepsilon)} ^{p/(p_\ast -p)}.$ Next we use the elementary inequality $(a+b)^\alpha  \leq a ^\alpha + \alpha (a+b)^{\alpha -1} b, $ $\alpha \geq 1,$ $a,$ $b>0,$ to get that
	\begin{equation*}
	\| \omega _\sigma \| _{(\varepsilon)} ^N \leq \| \nabla \omega _\sigma \| _p ^N + \frac{N}{p} V_ \varepsilon\left( \| \nabla \omega _\sigma \|^p_p  + V_\varepsilon \| \omega _\sigma \| _p ^p \right)^{(N-p)/p} \| \omega _\sigma \| _p ^p ,
	\end{equation*}
	where $V_\varepsilon = \max _{x \in B_{\varrho_\varepsilon}} V(\varepsilon x),$ and by \eqref{ozao1} we have
	\begin{align*}
	\| \nabla \omega _\sigma \| _p ^N &\leq \mathbb{S}^{N/p}+ \frac{N}{p}(\mathbb{S} +  \mathcal{O}(\sigma ^{(N-p) / p}))^{(N-p)/p}\mathcal{O}(\sigma ^{(N-p) / p})\\
	&=\mathbb{S}^{N/p}+ \mathcal{O}(\sigma ^{(N-p) / p}),
	\end{align*}
	for $0<\sigma <1$ sufficiently small. We conclude that
	\begin{equation*}
	c(I_\varepsilon) \leq \max _{t \geq 0} I_{\varepsilon} (t \omega _\sigma )   \leq \frac{1}{N}\mathbb{S}^{N/p}+\mathcal{O}(\sigma ^{(N-p) / p}) +C_\varepsilon\| \omega _\sigma \| _p ^p - \lambda b_1^{p_0} \| \omega _\sigma \|^{p_0}_{L ^{p_0} (B_{\varrho_\varepsilon})},
	\end{equation*}
	where $C_\varepsilon =  (1/p) V _\varepsilon (\mathbb{S} +V_\varepsilon )^{(N-p)/p}.$ Taking into account estimates \eqref{ozao1}--\eqref{ozao3} and \ref{pi}--\ref{piii} one can verify that $\mathcal{O}(\sigma ^{(N-p) / p}) + C_\varepsilon\| \omega _\sigma \| _p ^p - \lambda b_1^{p_0} \| \omega _\sigma \|^{p_0}_{L ^{p_0} (B_{\varrho_\varepsilon})} <0$ for $\sigma $ small enough.
\end{proof}
In the following results we use Lemma \ref{l_geometry} (see \cite{elves2009}) to guarantee the existence of a Palais-Smale sequence at the mountain pass level, that is, $(v_n) \subset X_\varepsilon$ such that $I_\varepsilon (v_n) \rightarrow c(I_\varepsilon)$ and $\| I' (v_n) \|_{X_\varepsilon ^\ast } \rightarrow 0.$ Additionally, from here on we suppose \ref{V_brutal} and the same conditions of Lemma \ref{l_minimax}.
	\begin{lemma}\label{l_psbounded}
	$(v_n)$ is bounded in $X_\varepsilon.$
	\end{lemma}
	\begin{proof}
		For $n$ large enough, by Proposition \ref{p_ge}--(ii) and (iii), we have
		\begin{align*}
		1+c(I_\varepsilon)+\| v_n \| _{(\varepsilon )} & \geq I_\varepsilon (v_n)-\frac{1}{\mu} I'_\varepsilon (v_n)\cdot v_n\\
		& \geq  \left( \frac{1}{p} - \frac{1}{\mu}  \right)     \| v_n \| _{(\varepsilon )} ^p + \int _{\Omega ^c_\varepsilon}\frac{1}{\mu}g (\varepsilon x , v_n) v_n - G (\varepsilon x , v_n) \dx\\
		& \geq \left( \frac{1}{p} - \frac{1}{\mu}  \right)\| v_n \| _{(\varepsilon )} ^p  + \left(\frac{p - \mu }{\mu } \right)\int _{\Omega ^c_\varepsilon} \frac{V_0}{p \kappa_0}(v_n^+)^p\dx \\
		& \geq \left( \frac{1}{p} - \frac{1}{\mu}  \right)\left[   \| \nabla v_n \| _{p} ^p + \left(1- \frac{1}{\kappa_0} \right) \int _{\mathbb{R}^N}   V(\varepsilon x) |v_n|^p\dx\right]\\
		&\geq\left( \frac{1}{p} - \frac{1}{\mu}  \right) \left( 1 - \frac{1}{\kappa_0}\right) \| v_n \| _{(\varepsilon )}^p.
		\end{align*}
		In particular, one can see that $(v_n)$ cannot be unbounded.
	\end{proof}
Since $X_\varepsilon$ is reflexive space, by Lemma \ref{l_psbounded} we have, up to a subsequence,
\begin{gather}
v_n \rightharpoonup v \text{ in }X_\varepsilon \text{ and in }L ^{p ^\ast }(\mathbb{R}^N),\nonumber \\
v_n \rightarrow v\text{ in }L ^\theta _{\loca} (\mathbb{R}^N) \text{ for any }1 \leq \theta < p _\ast,\label{local_comp}\\
 v_n (x) \rightarrow v (x) \text{ a.e. in }\mathbb{R}^N.\nonumber
\end{gather}
\begin{lemma}\label{l_critconv}
$v_n \rightarrow v$ in $L _{\loca}^{p_\ast } (\mathbb{R}^N),$ up to a subsequence.
\end{lemma}
\begin{proof}
Our argument is close related to the one used in \cite[Lemma 7]{lyberopoulos2018}. There are nonnegative measures $\mu _\ast$ and $\nu _\ast$ such that, up to a subsequence,
\begin{equation*}
| \nabla v_n | ^p \rightharpoonup \mu _\ast \quad \text{and}\quad |v _n |^{p ^\ast } \rightharpoonup \nu _\ast,
\end{equation*}
weakly in the space of measures $\mathcal{M}(\mathbb{R}^N).$ Therefore one can apply Lions' concentration-compactness principle \cite{lions_limitcase1}	to obtain the existence of an, at most, countable index set $\mathscr{J} \subset \mathbb{N}$ (possible empty), a collection of distinct points $\{ x_k\in \mathbb{R}^N: k \in \mathscr{J}\}$ and two set of numbers $\{ \mu _k  > 0: k \in \mathscr{J} \}$ and $\{  \nu _k  > 0: k \in \mathscr{J} \}$ such that
\begin{equation}\label{cclions}
\mu _\ast \geq | \nabla v | ^p + \sum _{k \in \mathscr{J} } \mu _k \delta _{x_k}\quad \text{and}\quad \nu _\ast = |v|^{p ^\ast } + \sum _{k \in \mathscr{J}  }\nu _k \delta _{x_k},
\end{equation}
where $\delta _{x_k}$ is the Dirac measure supported at $x_k.$ Furthermore,
\begin{equation}\label{cc_inequality}
\mu _k \geq \mathbb{S} \nu _k ^{p/p^\ast},\quad \forall \, k \in \mathscr{J},\quad \text{and}\quad \sum _{k \in \mathscr{J}  }\nu _k  < +\infty.
\end{equation}
We are going to prove that $\mathscr{J} = \emptyset,$ by arguing by contradiction and assuming the existence of $j \in \mathscr{J}.$ Let $\psi \in C^\infty([0, \infty):[0,1])$ such that $ \psi ' \in L^\infty ([0, \infty )),$ $\psi (t) = 1,$ if $0 \leq t \leq 1,$ and $\psi (t)=0,$ if $t\geq 2.$ Given $\delta >0,$ consider $\psi _{j,\delta } (x) := \psi (|x - x_j| ^2\delta ^{-2})$ and $A_{j,\delta } = B_{\sqrt{2} \delta }(x_j) \setminus \overline{B _\delta }(x_j).$ It is easy to see that $\psi _{j,\delta } \neq 0$ in $A_{j,\delta }$ and $\|\nabla \psi _{j,\delta} \|_\infty \leq 2\| \psi ' \| _\infty \delta ^{-1}.$ Furthermore, since $I'_\varepsilon (v_n) \cdot  (\psi _{j, \delta } v_n) = o_n (1),$
\begin{align}
\int_{\mathbb{R}^N} & v_n | \nabla v_n |^{p-2}\nabla v_n \cdot \nabla \psi _{j,\delta}  \dx + \int _{\mathbb{R}^N}  \psi _{j,\delta}\dmu _\ast  + \int _{\mathbb{R}^N}  V(\varepsilon x) |v| ^p \psi _{j,\delta} \dx \nonumber\\ &= \int _{\mathbb{R}^N} g(\varepsilon x, v_n) v_n \psi _{j,\delta} \dx + o_n (1).\label{idbasic2}
\end{align}
where we have used \eqref{local_comp}. Now we show that
\begin{equation}\label{gradiente_delta}
\lim _{\delta \rightarrow 0} \left[ \limsup _{n \rightarrow \infty} \int_{\mathbb{R}^N}  v_n | \nabla v_n |^{p-2}\nabla v_n \cdot \nabla \psi _{j,\delta}  \dx \right] = 0.
\end{equation}
In fact, by H\"{o}lder's inequality we have
\begin{align}
\left| \int _{\mathbb{R}^N}  v_n | \nabla v_n |^{p-2}\nabla v_n \cdot \nabla \psi _{j,\delta}  \dx   \right| & \leq |A_{j, \delta }|^{1/N}\left( \int _{A_{j, \delta }}  |\nabla v_n |^p \dx  \right)^{(p-1)/p} \left( \int _{A_{j, \delta}}  |v_n| ^{p ^\ast } |\nabla \psi_{j,\delta} |^{p ^\ast} \dx   \right)^{1/p^\ast } \nonumber\\
& \leq C \left(  \int _{A_{j, \delta } } |v| ^{p ^\ast}\dx + \sum _{k \in \mathscr{J}_{j, \delta } }\nu _k  + o_n (1) \right)^{1/ p ^\ast}, \label{idbasic3}
\end{align}
where $\mathscr{J}_{j, \delta } = \{ x_k \in \mathscr{J} : k \in A_{j, \delta } \}.$ Notice that if $\mathscr{J}_{j, \delta }$ is finite, we can take $\delta$ small enough in such way that $\mathscr{J}_{j, \delta } = \emptyset.$ Suppose $\mathscr{J}_{j, \delta }$ infinite and denote $m_{j, \delta } = \min \mathscr{J}_{j, \delta }.$ Then $m_{j, \delta } \rightarrow \infty ,$ as $\delta \rightarrow 0.$ In fact, if not, there would exist $M_0>0$ and a sequence $\delta _n \rightarrow 0$ with $m_{j, \delta _n} \leq M_0.$ Hence $(m_{j, \delta _n})_n$ is finite and there is a subsequence such that $m_{j, \delta _{n_k}} = m_{j, \delta _{l_0}}$ for some $l_0 \in \mathbb{N}.$ This implies that $m_{j, \delta _{n_k}} \in \mathscr{J}_{j, \delta _{l_0}}$ and so $x_{n_k} \in A_{j, \delta _{l_0}}$ for any $k,$ which leads to a contradiction with $x_{n_k} \rightarrow x_j.$ Since $\lim _{\delta \rightarrow 0}m_{j, \delta } = \infty$  and $\sum _{k \in \mathscr{J}  }\nu _k$ converges, $\sum _{k \in \mathscr{J}_{j, \delta } }\nu _k \rightarrow 0, $ as $\delta \rightarrow 0.$ This fact together with dominated converge theorem allow us to obtain \eqref{gradiente_delta}, by taking $\delta \rightarrow 0$ in \eqref{idbasic3}.

On the other hand, applying \eqref{local_comp} again together with \eqref{cclions} we obtain
\begin{align}
\int _{\mathbb{R}^N } g (\varepsilon x, v_n) v_n \psi _{j, \delta } \dx \leq  & \int _{ \Omega _\varepsilon } f(v) v \psi _{j, \delta } \dx +  \frac{1}{\kappa_0}\int _{ \Omega _\varepsilon ^c } V(\varepsilon x ) |v| ^p \psi _{j, \delta } \dx \nonumber \\
& +\int _{\mathbb{R}^N } |v| ^{p ^\ast } \psi _{j, \delta } \dx + \sum _{k \in \mathscr{J}} \nu _k \psi _{j, \delta  } (x_k)+ o_n(1).\label{deltalimite}
\end{align}
Clearly $\sum _{k \in \mathscr{J}} \nu _k \psi _{j, \delta  } (x_k)$ is uniformly convergent on $\delta$ and we can take the limit as $\delta \rightarrow 0$ in \eqref{deltalimite} to get
\begin{equation}\label{gzinho}
\lim _{\delta \rightarrow 0 }\left[\limsup _{n \rightarrow \infty}\int _{\mathbb{R}^N } g (\varepsilon x, v_n) v_n \psi _{j, \delta } \dx  \right] \leq \nu _j,
\end{equation}
where dominated converge theorem is applied. Next we have
\begin{equation}\label{mizao}
\int _{\mathbb{R}^N}  \psi _{j,\delta}\dmu _\ast \geq  \int _{\mathbb{R}^N } |\nabla u|^p  \psi _{j,\delta }\dx + \sum _{k \in \mathscr{J }} \mu _k \psi _{j,k} (x_k) \geq \mu _j.
\end{equation}
Nevertheless, since $|\nabla v_n|^p \rightharpoonup \mu _\ast $ in the sense of measures, we known
\begin{equation}\label{weak_measure}
\mu_\ast  (\mathbb{R}^N) \leq \liminf _{n \rightarrow \infty } \int _{\mathbb{R}^N }|\nabla v_n| ^p \dx,
\end{equation}
implying in
\begin{equation}\label{delta_conv}
\lim _{\delta \rightarrow 0 }\int _{\mathbb{R}^N} \psi_{j, \delta  } \dmu_\ast = 0.
\end{equation}
Taking into account \eqref{gradiente_delta}, \eqref{gzinho}, \eqref{mizao} and \eqref{delta_conv} to pass the limit first as $n \rightarrow \infty $ and then $\delta \rightarrow 0$ in \eqref{idbasic2}, we conclude that $\mu _j \leq \nu _j.$ This fact together with \eqref{cc_inequality} leads to $\mu _j \geq \mathbb{S}^{N/p}.$ By \eqref{weak_measure} and from the proof of Lemma \ref{l_psbounded}, there holds
\begin{align*}
c(I_\varepsilon ) & \geq\left( \frac{1}{p} - \frac{1}{\mu}  \right) \left( 1 - \frac{1}{\kappa _0} \right) \liminf_{n \rightarrow \infty }\int _{\mathbb{R}^N} |\nabla v_n|^p \dx \\
& \geq \frac{1}{p} \left[ \int _{\mathbb{R}^N} | \nabla v | ^p \dx  + \sum _{k \in \mathscr{J} } \mu _k  \right] \geq \frac{1}{N}\mathbb{S}^{N/p}.
\end{align*}
This inequality contradicts Lemma \ref{l_minimax} and allow us to conclude that $\mathscr{J} = \emptyset.$ To complete the proof let us take $\eta _R \in C^\infty _0 (\mathbb{R}^N:[0,1])$ such that $\eta =1$ in $B_R$ and $\eta = 0 $ in $B_{2R}^c.$ Recalling \eqref{local_comp} and \eqref{cclions}, we see that
\begin{equation*}
v_n \eta _R ^{1/p ^\ast } \rightharpoonup v \eta _R ^{1/p ^\ast }\text{ in }L ^{p ^\ast } (\mathbb{R}^N)\quad \text{and}\quad \| v_n \eta _R ^{1/p ^\ast } \| _{p ^\ast } \rightarrow \| v \eta _R ^{1/p ^\ast } \| _{p ^\ast },
\end{equation*}
up to a subsequence. The fact that $L^{p ^\ast } (\mathbb{R}^N)$ is a uniformly convex space leads to $\| (v_n - v) \eta _R ^{1/p ^\ast } \|_{p ^\ast } \rightarrow 0.$
\end{proof}
\begin{lemma}\label{l_delpino}
	For each $ \delta >0,$ there is $R=R(\varepsilon,  \delta ) >0$ such that
	\begin{equation*}
	\limsup _{n \rightarrow \infty} \int _{ B^c_R} |\nabla v_n |^p +V(\varepsilon x)|v_n|  ^p \dx < \delta.
	\end{equation*}
\end{lemma}
\begin{proof}
	Let $B_{R/2} \supset \Omega _\varepsilon$ and define $\psi _R \in C_0 ^\infty (\mathbb{R}^N)$ such that $\psi _R(x) = 0 ,$ if $x \in B_{R/2},$  and $\psi _R(x) = 1,$ if $x \not \in B_R,$ with $\|\nabla \psi _R \|_\infty  \leq C/R,$ for some $C>0.$ Since $(\psi _R v_n)_n \subset X_\varepsilon$ is bounded, we have $I' (v_n)\cdot (\psi _R v_n) = o_n(1)$ and 
	\begin{equation*}
	\int _{\Omega^c _\varepsilon} \psi _R \left(  |\nabla v_n| ^p + V(\varepsilon x)|v_n| ^p  \right)\dx = o_n(1) - \int _{\Omega^c _\varepsilon} v_n |\nabla v_n|  ^{p-2} \nabla v_n \cdot \nabla \psi _R \dx + \int _{\Omega^c _\varepsilon} \psi _R g(\varepsilon x , v_n)v_n \dx.
	\end{equation*}
	On the other hand, by Proposition \ref{p_ge}--(iii), if $\varepsilon x \not \in \Omega$ then $g(\varepsilon x, t ) \leq  (V _0 / \kappa_0)t^{p-1} \leq (1/\kappa_0) V(\varepsilon x)  t^{p-1}.$ Hence
	\begin{equation*}
		\int _{\Omega _\varepsilon ^c} |\nabla v_n| ^p \psi _R  \dx + \left(1-\frac{1}{\kappa_0} \right) \int _{\Omega _\varepsilon ^c} V(\varepsilon x )|v_n|^p \psi _R	\dx \leq o_n(1)	 - \int _{ \Omega _\varepsilon ^c} v_n |\nabla v_n| ^{p-2} \nabla v_n \cdot \nabla \psi _R \dx
	\end{equation*}
Furthermore, H\"{o}lder's inequality implies,
	\begin{equation*}
	\left| \int _{\Omega _\varepsilon ^c}  v_n |\nabla v_n| ^{p-2}\nabla v_n \cdot \nabla \psi _R \dx \right| 	\leq \| \nabla \psi _R \| _\infty \| \nabla v_n \| _p \| v_n \|_{L^p(\Omega _\varepsilon ^c)}.
	\end{equation*}
In particular, since $(\|v_n\|_{(\varepsilon)})_n$ is bounded and \eqref{usa}, we have
\begin{equation}\label{final1}
    V_0 \int _{\Omega _\varepsilon ^c} |v_n|^p\dx \leq \int _{\Omega _\varepsilon ^c} V(\varepsilon x ) |v_n|^p \dx \leq \| v_n \| _{(\varepsilon)}^p ,
\end{equation}
and
	\begin{equation*}
	\left(1- \frac{1}{\kappa_0} \right)\int _{B^c_{R}}  |\nabla v_n| ^p + V(\varepsilon x)|v_n|^p  \dx \leq o_n(1) + \frac{C}{R}.\qedhere
	\end{equation*}
\end{proof}
\begin{lemma}
$v_n \rightarrow v$ in $X_\varepsilon,$ up to a subsequence.
\end{lemma}
\begin{proof}
Recalling $v_n\rightharpoonup v$ in $X_\varepsilon$ and $V(\varepsilon x ) = V(\varepsilon x )^{1/p+(p-1)/p}$ it is clear that
\begin{equation*}
\int _{\mathbb{R}^N } |\nabla v|^{p-2}\nabla v \cdot \nabla (v_n - v) \dx = o_n(1)\quad \text{and} \quad \int _{\mathbb{R}^N} V(\varepsilon x) |v|^{p-2}v \, (v_n - v) \dx = o_n (1).
\end{equation*}
We apply this fact on the expression of $I'_\varepsilon(v_n) \cdot (v_n - v)=o_n(1)$ to obtain
\begin{align*}
o_n(1) &= I'_\varepsilon(v_n) \cdot (v_n - v)\\
& = \int _{\mathbb{R}^N } |\nabla v_n| ^{p-2} \nabla v_n \cdot \nabla ( v_n - v) \dx + \int _{\mathbb{R}^N } V(\varepsilon x)  |v_n | ^{p-2}  v_n (v_n - v)\dx - \int _{\mathbb{R}^N} g(\varepsilon x, v_n) (v_n - v ) \dx\\
& = \int _{\mathbb{R}^N } (|\nabla v_n| ^{p-2} \nabla v_n - |\nabla v|^{p-2}\nabla v ) \cdot  \nabla (v_n - v)\dx + \int _{\mathbb{R}^N } V(\varepsilon x) (|v_n|^{p-2}v_n - |v|^{p-2}v) (v_n - v)\dx \\
& \qquad  - \int _{\mathbb{R}^N} g(\varepsilon x, v_n) (v_n - v ) \dx+ o_n(1).
\end{align*}
It suffices to prove that
\begin{equation}\label{gefinal}
\int _{\mathbb{R}^N} g(\varepsilon x, v_n) (v_n - v ) \dx=o_n(1).
\end{equation}
In fact, applying the following inequality \cite[Proposition 8]{lyberopoulos2018},
\begin{align*}
\int _{\mathbb{R}^N } & \varrho (x) (|X|^{p-2}X - |Y|^{p-2}Y)\cdot (X-Y) \dx \\  & \geq \left( \|\varrho ^{1/p} |X| \|^{p-1} _p - \| \varrho ^{1/p} |Y| \|^{p-1} _p  \right) \left(\|\varrho ^{1/p} |X| \|_p - \|  \varrho ^{1/p} |Y|  \|_p  \right),
\end{align*}
which holds for any $X,$ $Y:\mathbb{R}^N \rightarrow \mathbb{R}^m,$ $m \geq 1,$ functions in $L ^p (\mathbb{R}^N; \varrho)$ (Weighted Lebesgue spaces), we see that if \eqref{gefinal} yields, then $\| v_n \| _{(\varepsilon )} \rightarrow \| v \| _{(\varepsilon)},$ as $n \rightarrow \infty,$ and the fact that $X_\varepsilon $ is uniformly convex leads to $v_n \rightarrow v$ in $X_\varepsilon.$ To prove \eqref{gefinal} we start by noticing that H\"{o}lder's inequality, Proposition \ref{p_ge}--(iii) and Lemma \ref{l_delpino} implies
\begin{equation*}
\left| \int _{ B_R ^c } g(\varepsilon x, v_n) (v_n - v ) \dx \right| \leq\frac{1}{k_0} \left( \int _{B_R ^c} V(\varepsilon x ) |v_n| ^p  \dx \right)^{(p-1)/p} \left( \int _{B_R ^c} V(\varepsilon x)|v_n - v|^p \dx\right)^{1/p}<\delta,
\end{equation*}
for $R>0$ sufficiently large enough with $B_R \supset \Omega _\varepsilon$. Denoting $Y_n = \Omega _\varepsilon \cap [v_n \leq a_0]$ and in view of \eqref{local_comp} we get
\begin{align*}
\left| \int _{ B_R  } g(\varepsilon x, v_n) (v_n - v ) \dx \right| & \leq \int _{B_R \cap Y_n} (f(v_n) + |v_n| ^{p ^\ast -1})|v_n -v| \dx + o_n(1)\\
& \leq C \| v_n - v \| _{L ^{p ^\ast } (B_R)} + o_n(1) = o_n(1),
\end{align*}
where Lemma \ref{l_critconv} was applied.
\end{proof}
Clearly $o_n (1) = I ' _\varepsilon (v_n) \cdot v_n ^- = \|v_n ^- \|^p _{(\varepsilon)}$ and we have $v \geq 0$ a.e. in $\mathbb{R}^N.$ Summing up all the previous results we conclude:
\begin{proposition}
For any $\varepsilon >0, $ there exists $v_\varepsilon \in W^{1,p}(\mathbb{R}^N)$ a nontrivial nonnegative weak solution of the equation in \eqref{PF}, at the mountain pass level $I_\varepsilon (v_\varepsilon) = c(I_\varepsilon).$	
\end{proposition}
\subsection{Regularity and $L^\infty$ estimates}
Here we provide regularity results of the solutions $v_\varepsilon$ and also a uniform bound in the $L^\infty$--norm. In our argument we use a variant of the well known result due to H. Brezis and T. Kato \cite{breziskato,struwe}, which is usually applied for problems like \eqref{P} that involves a nonliterary with critical growth.
	\begin{lemma}\label{l_breziskato}
	Let $h:\mathbb{R}^N \times \mathbb{R} \rightarrow \mathbb{R}$ be a Carath\'{e}odory function satisfying the following property: for each $v \in W^{1,p}(\mathbb{R}^N)$ there is $a \in L^{N/p} (\mathbb{R}^N)$ such that
	\begin{equation}\label{estimativa}
		|h(x,v)| \leq (M_{h} + a(x))|v|^{p-1},\quad \text{in}\quad \mathbb{R}^N,
	\end{equation}
for some $M_h \geq 0.$ Then for any weak solution $v \in W^{1,p}(\mathbb{R}^N)$ of
\begin{equation}\label{eq_breziskato}
	- \Delta _p v = h(x,v)\quad \text{in}\quad \mathbb{R}^N,
\end{equation}
we have $v \in L ^{s }(\mathbb{R}^N),$ for any $s \in [p, \infty ).$ Furthermore, if instead $v \in W^{1,p}_{\loca}(\mathbb{R}^N),$ $a \in L^{N/p}_{\loca }(\mathbb{R}^N)$ and \eqref{estimativa} holds in balls, then $v \in L ^{s }_{\loca}(\mathbb{R}^N),$ for any $s \in [p, \infty ).$
	\end{lemma}
	\begin{proof}
		For $L,$ $s>0$ let us consider the real function $\psi_{L,s} (t) = t t_{L,s},$ where $t _{L,s} := \min \{ |t|^s , L  \}.$ Then $\psi'_{L,s}  \in L^\infty ([0, \infty )).$ Following, take $\phi \in C^\infty(\mathbb{R}^N:[0,1])$ with $\nabla \phi \in L^\infty (\mathbb{R}^N).$ Let us assume initially that $v \in W^{1,p}(\mathbb{R}^N).$ Clearly $vv^\theta  _{L,s} \phi^\theta \in W^{1,p}(\mathbb{R}^N),$ for $\theta \geq 1.$ Moreover,
		\begin{equation*}
		\nabla (v v^\theta _{L,s}\phi ^\theta ) =
		\left\{  
		\begin{aligned}
		&(s\theta +1)v^\theta  _{L,s}  \phi ^{\theta} \nabla v + \theta v v^{\theta}  _{L,s} \phi ^{\theta -1}\nabla \phi , &&\text{ if }|v|^s \leq L,\\
		&v^\theta _{L,s} \phi ^\theta \nabla v + \theta v v^\theta_{L,s}\phi ^{\theta -1} \nabla \phi , && \text{ if } |v|^s > L.
		\end{aligned}
		\right.
		\end{equation*}
		Taking $v v^p _{L,s}\phi^p $ as a test function for \eqref{eq_breziskato} we have
		\begin{align}
		\int _{\mathbb{R}^N} v_{L,s}^p  |\nabla v|^p \phi ^p \dx & + sp \int _{[|v|^s \leq L]} v_{L,s}^p  |\nabla v|^p \phi ^p \dx \nonumber  \\&= \int _{\mathbb{R}^N}h(x,v)(v v^p_{L,s} \phi ^p) \dx - p\int _{\mathbb{R}^N}v v_{L,s}^p \phi ^{p-1} |\nabla v|^{p-2} \nabla v \cdot \nabla \phi  \dx.\label{BK_um}
		\end{align}
		Now we estimate the right-hand side of \eqref{BK_um}. The first term can be estimated using \eqref{sobolev_constant}, \eqref{estimativa} and H\"{o}lder's inequality in the following way
		\begin{align}
		\int _{\mathbb{R}^N}h(x,v)(v v^p_{L,s} \phi ^p) \dx & \leq (M_h + l )\int _{\mathbb{R}^N} |v|^{(s+1)p} \phi ^p\dx + \gamma(l)\left(\int _{[|a(x)| > l ]}   |v v_{L,s} \phi |^{p^\ast}  \right)^{p/p^\ast}\nonumber\\ 
		& \leq (M_h + l )\int _{\mathbb{R}^N} |v|^{(s+1)p} \phi ^p\dx + \gamma (l) \mathbb{S}^{-1} \| \nabla (v v_{L,s} \phi ) \|^p_p,\label{BKum}
		\end{align}
		where $s$ is taken in a such way that $\int _{\mathbb{R}^N } |v|^{(s+1)p}\dx < +\infty$ and 
		\begin{equation*}
		\gamma(l) = \left( \int _{[ |a(x)| > l  ] } |a(x)|^{N/p}  \dx \right)^{p/N}  \rightarrow 0,\quad\text{as} \ l\rightarrow  + \infty.
		\end{equation*}
		One can use Young's inequality to estimate the second term in the right-hand side of \eqref{BK_um},
		\begin{equation}\label{BKdois}
		- p\int _{\mathbb{R}^N}v v_{L,s}^p \phi ^{p-1} |\nabla v|^{p-2} \nabla v \cdot \nabla \phi  \dx \leq p \delta \int _{\mathbb{R}^N } v_{L,s}^p |\nabla v|^p \phi ^p \dx + p C_\delta \int _{\mathbb{R}^N}  |v v_{L,s} \nabla \phi |^p \dx,
		\end{equation}
		where $\delta >0$ is such that $1-p\delta \geq 1/2.$ We use estimates \eqref{BKum} and \eqref{BKdois} in \eqref{BK_um} to get
		\begin{align}
\frac{1}{2}\int _{\mathbb{R}^N} v_{L,s}^p &|\nabla v|^p \phi ^p\dx  + sp \int _{[|v|^s \leq L] } v_{L,s}^p  |\nabla v|^p \phi ^p\dx\nonumber		\\ & \leq (M_h + l )\int _{\mathbb{R}^N} |v|^{(s+1)p} \phi ^p\dx + \gamma (l) \mathbb{S}^{-1} \| \nabla (v v_{L,s} \phi ) \|^p_p + p C_\delta \int _{\mathbb{R}^N}  |v v_{L,s} \nabla \phi |^p \dx  \label{KBtres}
		\end{align}
On the other hand,
\begin{align}
\int _{\mathbb{R}^N}|\nabla (vv_{L,s}\phi )|^p  \dx  \leq C_\ast & \left[\frac{1}{2}\int _{\mathbb{R}^N}  v_{L,s}^p |\nabla v|^p \phi ^p\dx \right.\nonumber\\
&\quad  \left. + sp \int _{[|v|^s \leq L] } v_{L,s}^p  |\nabla v|^p \phi ^p\dx	+ \int _{\mathbb{R}^N} |v v_{L,s} \nabla \phi |^p \dx \right],\label{BK_tres}
\end{align}
where $C_\ast = \max\{ 2^p (s+1)^p, 2^{p-1}/(sp) , 2^{p-1}  \}.$ Fixing $l>0$ such that $1- C_\ast \mathbb{S}^{-1} \gamma (l) > 0 $ and using \eqref{KBtres} in \eqref{BK_tres} we have
\begin{equation*}
(1-C_\ast \gamma (l) \mathbb{S}^{-1})  \| \nabla (v v_{L,s} \phi ) \|_p ^p\leq C_\ast \left[ (M_h + l)\int _{\mathbb{R}^N } |v|^{(s+1)p} \phi ^p\dx +(1+p C_\delta ) \int _{\mathbb{R}^N }|v|^{(s+1)p} |\nabla\phi|^p \dx \right].
\end{equation*}
Now we use \eqref{sobolev_constant} and apply Fatou's lemma by taking $L\rightarrow + \infty$ to conclude
\begin{align}
(1-C_\ast \gamma (l) \mathbb{S}^{-1}) & \mathbb{S} \left( \int _{\mathbb{R}^N}|v|^{(s+1)p^\ast}\phi^{p^\ast}   \dx \right) ^{p/p^\ast} \nonumber\\
& \leq C_\ast\left[ (M_h + l)\int _{\mathbb{R}^N } |v|^{(s+1)p} \phi ^p\dx + (1+p C_\delta )\int _{\mathbb{R}^N }|v|^{(s+1)p} |\nabla\phi|^p \dx \right].\label{combo1}
\end{align}
Notice that inequality \eqref{combo1} holds for any $\phi \in C^\infty(\mathbb{R}^N:[0,1])$ with $\nabla \phi \in L^\infty (\mathbb{R}^N).$ Furthermore, if instead $v \in W^{1,p}_{\loca}(\mathbb{R}^N),$ clearly \eqref{combo1} still holds whenever $\phi \in C^\infty _0(\mathbb{R}^N).$ In this case we can choose $s=s_1$ such that $(s_1+1)p=p^\ast$ to obtain, from \eqref{combo1}, that $|v|^{s_1+1} \phi \in L ^{p^\ast }(\mathbb{R}^N), $ for any $\phi \in C^\infty _0(\mathbb{R}^N).$ Hence $|v|^{s_1+1} \in L ^{p^\ast }(B_R), $ for a given $R>0.$ Now we take $s_2>0$ with $(s_2+1)p = (s_1 + 1)p^\ast,$ to get $|v|^{s_2+1}\phi \in L ^p (\mathbb{R}^N).$ Applying again \eqref{combo1} we see that $|v|^{s_2+1}\phi \in L^{p ^\ast }(\mathbb{R}^N),$ for all $\phi \in C^\infty _0 (\mathbb{R}^N).$ Thus $|v|^{s_2 +1 } \in L^{p^\ast}(B_R),$ for any $R>0.$ This leads to an iteration procedure in $s$ which implies that $v  \in L_{\loca}^{s_i} (\mathbb{R}^N),$ where $s_i = (p^\ast /p)^i - 1,$ for all $i \in \mathbb{N}.$ Conclusion that $v\in L^s_{\loca}(\mathbb{R}^N),$ follows by an interpolation inequality. If $v \in W^{1,p}(\mathbb{R}^N),$ then the same iteration argument applies by taking $\phi =1.$
\end{proof}
	Next we state a extension for the $p$--Laplacian operator of the well known regularity result due to C. O. Alves and M. A. Souto \cite[Proposition 2.6]{alvessouto2012} for solutions of semilinear problems. This result is typically applied on the study of equations like \eqref{P} ($\varepsilon=1$) in the presence of vanishing potentials.
	\begin{lemmaletter}\cite[Lemma 2.4]{liyang2013}\label{l_chines}
		Let $Q \in L^{s } (\mathbb{R}^N),$ $s >N/p,$ and $v \in W^{1,p}(\mathbb{R}^N)$ be a weak solution of the equation
		\begin{equation*}
		-\Delta _p v + b(x) |v|^{p-2}v =h_0(x,v),\quad \text{in}\quad \mathbb{R}^N,
		\end{equation*}
		where $b \geq 0$ is measurable and $h_0:\mathbb{R}^N \times \mathbb{R} \rightarrow \mathbb{R}$ is a continuous function satisfying
		\begin{equation*}
		|h_0(x,t)|\leq Q (x) |t|^{p-1},\quad \forall \, (x,t)\in \mathbb{R}^N\times \mathbb{R}.
		\end{equation*}
		Then $v \in L^\infty (\mathbb{R}^N)$ and there is  a constant $M = M (s,\| Q\| _s) >0$ such that $\| v \| _\infty \leq M \| v \| _{p^\ast }.$
	\end{lemmaletter}
	Notice that from the proof of Lemma \ref{l_psbounded} we have
	\begin{equation}\label{norma_est}
		\| v _ \varepsilon \|^p _{(\varepsilon)} \leq c(I_\varepsilon) \left[ \left(\frac{1}{p} - \frac{1}{\mu} \right) \left( 1 - \frac{1}{\kappa_0}\right) \right]^{-1},\quad \forall \ \varepsilon\in (0,1).
	\end{equation}
	Hence, by Lemma \ref{l_minimax}, a uniform bound for $(\|v_\varepsilon\|_{(\varepsilon)})_{0<\varepsilon < 1}$ is obtained. Particularly, 
	since $\mathbb{S} \| v_\varepsilon\|_{p^\ast }^p \leq  \| \nabla v_\varepsilon \|^p_p,$	
	\begin{equation}\label{salvation}
	\left(\int _{\mathbb{R}^N}|v_\varepsilon |^{p^\ast} \dx \right)^{p/p^\ast}< \frac{1}{N}\mathbb{S}^{(N-p)/p}\left[ \left(\frac{1}{p} - \frac{1}{\mu} \right) \left( 1 - \frac{1}{\kappa_0}\right) \right]^{-1},\quad \forall \ \varepsilon\in (0,1).
	\end{equation}
	Moreover, using Lemmas \ref{l_breziskato} and \ref{l_chines} we have the following conclusion, where condition $q_0>p$ in \ref{f_brutal} is necessary, when $\inf_{\xi \in \mathbb{R}^N}V(\xi) = 0$ (see \cite{byeon-wang2002,byeon-wang2003,lyberopoulos2018}).
		\begin{proposition}\label{p_linf}
		\begin{enumerate}[label=(\roman*)]
			\item $v_\varepsilon \in L^s (\mathbb{R}^N),$ for all $s \in [p, \infty ];$
			\item There is $M_\ast>0$ such that $\| v_\varepsilon \| _\infty \leq M_\ast,$ for all $\varepsilon \in (0,1);$
			\item $v_\varepsilon \in  C^{1,\alpha }_{\loca}(\mathbb{R}^N)$ and $v_\varepsilon >0$ in $\mathbb{R}^N$ for any $\varepsilon \in (0,1).$
		\end{enumerate}
	\end{proposition}
\begin{proof} We prove the case where $\inf_{\xi \in \mathbb{R}^N}V(\xi) = 0,$ see Remark \ref{r_infpositivo} below.

(i) First we notice that $v_\varepsilon$ is a weak solution of $- \Delta _p v_\varepsilon = h(x,v_\varepsilon)$ in $\mathbb{R}^N,$ where $h(x,t)=g(\varepsilon x, t ) - V(\varepsilon x) t ^{p-1}.$ By Remark \ref{r_growth}--(i)--(b) and Proposition \ref{p_ge}--(i) one can see that
\begin{equation*}
	|h( x,v)| \leq ((c_f\epsilon_0 + V(\varepsilon x)) + c_f(C_0 + 1) |v|^{p^\ast -p} )|v|^{p-1}\leq(M_0+a_0(x))|v|^{p-1}\quad\text{in}\quad B_R,
\end{equation*}
where $M_0 = c_f\epsilon_0 + \sup _{x \in B_R}V(\varepsilon x)$ and $a_0(x) =c_f(C_0 + 1) |v|^{p^\ast -p},$ for any $R>0,$ and $v \in W^{1,p}( \mathbb{R}^N).$ Since $a_0\in L^{N/p}(\mathbb{R}^N)$ we can apply Lemma \ref{l_breziskato} to obtain that $v_\varepsilon\in L^s_{\loca}(\mathbb{R}^N),$ for any $s \in [p, \infty).$ Next we prove that $v_\varepsilon \in L^s(B_R^c)$, for all $s \in [p, \infty)$ and $R>0.$ Let us apply Remark \ref{r_growth}--(i)--(b) and Proposition \ref{p_ge}--(i) once again to write
\begin{equation*}
	\int_{\mathbb{R}^N} |\nabla v_\varepsilon|^{p-2} \nabla v_\varepsilon \cdot \nabla \varphi +V(\varepsilon x) v_\varepsilon ^{p-1}\varphi \dx = \int _{\mathbb{R}^N} g(\varepsilon x,v_\varepsilon ) \varphi \dx \leq \int _{\mathbb{R}^N } c_f \left( \epsilon_0 v_\varepsilon^{p-1} + C_0 v_\varepsilon ^{p^\ast -1}  \right)\varphi \dx,
\end{equation*}
for any $\varphi \in X_\varepsilon$ with $\varphi \geq 0.$ Consequently,
\begin{equation*}
	\int_{\mathbb{R}^N} |\nabla v_\varepsilon|^{p-2} \nabla v_\varepsilon \cdot \nabla \varphi +(V(\varepsilon x) -  c_f \epsilon_0  )v_\varepsilon ^{p-1}\varphi \dx \leq c_f C_0 \int _{\mathbb{R}^N }  v_\varepsilon ^{p^\ast -1}  \varphi \dx.
\end{equation*}
Now we take $\varphi = v_\varepsilon (v_\varepsilon)^p _{L,s} \phi ^p$ as the test function in the proof of Lemma \ref{l_breziskato}, choosing $\phi\in C^\infty(\mathbb{R}^N:[0,1])$ with $\nabla \phi \in L^\infty (\mathbb{R}^N)$ and $R>0$ in a such way that 
\begin{equation}\label{propfi}
\phi =0 \text{ in } B_{R/2}\supset \Omega _\varepsilon.
\end{equation}
We also choose $\epsilon_0 = V_0/c_f,$ in order to obtain $V(\varepsilon x) -  c_f \epsilon_0  >0$ in $B^c_R$ (see \eqref{usa}) and so
\begin{equation*}
	\int_{\mathbb{R}^N} |\nabla v_\varepsilon|^{p-2} \nabla v_\varepsilon \cdot \nabla (v_\varepsilon (v_\varepsilon)^p _{L,s} \phi ^p) \dx \leq  \int _{\mathbb{R}^N } a_\varepsilon(x) v_\varepsilon^{p}(v_\varepsilon)^p _{L,s} \phi ^p\dx,
\end{equation*}
where $a_\varepsilon(x) = c_f C_0 v_\varepsilon ^{p^\ast - p }\in L^{N/p}(\mathbb{R}^N).$ From this point on, one can follow the same idea made on the proof of Lemma \ref{l_breziskato}, taking $h(x,t) = a_\varepsilon(x)|t|^{p-1},$ to conclude that inequality \eqref{combo1} holds with $v=v_\varepsilon$ and for any $\phi$ under \eqref{propfi}. Thus, the same argument holds: $v_\varepsilon ^{s_n+1} \phi  \in L^{p^\ast } (B^c_{R}),$ where $(s_n+1)p = (s_{n-1}+1)p^\ast,$ $s_0:=0,$ $\phi\in C^\infty(\mathbb{R}^N:[0,1]),$ under \eqref{propfi} with $\nabla \phi \in L^\infty (\mathbb{R}^N).$ In particular, $v_\varepsilon ^{s_n + 1} \in L^{p^\ast } (B^c_{R}),$ for any $R>0$ and $n \in \mathbb{N}.$ Hence $v_\varepsilon \in L ^s(\mathbb{R}^N),$ for any $s \in [p,\infty).$

(ii) We use part (i). It is clear that $v_\varepsilon$ is a weak solution of 
\begin{equation}\label{eq_ajuda}
-\Delta v_\varepsilon + b(x)v_\varepsilon^{p-1} = h_0 (x,v_\varepsilon)\quad\text{in}\quad\mathbb{R}^N,
\end{equation}
where $b(x) = V(\varepsilon x)$ and $h_0(x,t) = Q_\varepsilon (x)t^{p-1},$ for $t\geq 0,$ $h_0(x,t) = 0 ,$ for $t\leq 0,$ with
\begin{equation*}
Q_\varepsilon(x) =
\left\{  
\begin{aligned}
&\frac{g(x,v_\varepsilon)}{v_\varepsilon ^{p-1}}, &&\text{ if }v_\varepsilon (x) >0,\\
&0, && \text{ if } v_\varepsilon (x)=0.
\end{aligned}
\right.
\end{equation*}
By Remark \ref{r_growth}--(i)--(c) and Proposition \ref{p_ge}--(i) there are $C_1,$ $C_2>0$ such that
\begin{equation*}
Q_\varepsilon (x) \leq C_1 v_\varepsilon ^{q_0 - p } + C_2 v_\varepsilon ^{p^\ast - p }.
\end{equation*}
On the other hand, the iteration method used above in \eqref{combo1} together with \eqref{salvation}, leads to
\begin{equation}\label{salva_inter}
\int _{\mathbb{R}^N} |v_\varepsilon |^{(s_n +1)p^\ast}\dx \leq C_n,\quad   \forall \ \varepsilon\in (0,1),\ \forall \  n \in \mathbb{N},
\end{equation}
for some $C_n >0.$ Taking $k<m<n$ such that
\begin{equation*}
s_k > \max\{N/p,p  \}\quad \text{and}\quad s_k < (q_0 - p) s_m < (p^\ast - p)s_m < s_n,
\end{equation*}
we have, by an interpolation inequality and \eqref{salva_inter},
\begin{equation*}
\|v_\varepsilon \| _{(q_0-p)s_m} \leq \| v_\varepsilon \|^{1-\alpha}_{s_k} \| v_\varepsilon \|^{\alpha}_{s_n} \leq C'_{k,n} \quad\text{and} \quad\|v_\varepsilon \| _{(p^\ast-p)s_m} \leq \| v_\varepsilon \|^{1-\beta}_{s_k} \| v_\varepsilon \|^{\beta}_{s_n}\leq C''_{k,n},
\end{equation*}
for suitable $\alpha,$ $\beta \in [0,1]$ and any $\varepsilon \in (0,1).$ Thus $Q_\varepsilon \in L^s(\mathbb{R}^N),$ for $s=s_m >N/p$ and $\| Q_\varepsilon\|_s \leq C'''_{k,n},$ for all $\varepsilon \in (0,1).$ Applying Lemma \ref{l_chines}, we have $v_\varepsilon \in L^\infty (\mathbb{R}^N) $ and the existence of $M>0$ that does not depend on $\varepsilon$ such that $\| v_\varepsilon \| _\infty \leq M \| v_\varepsilon \| _{p^\ast }$ (see \cite[Proof of Lemma 2.4--(2.29)]{liyang2013}). Using \eqref{salvation} we have the uniform bound.

(iii) The fact that  $v_\varepsilon \in C^{1,\alpha }_{\loca} (\mathbb{R}^N)$ follows by the regularity results of \cite{dibenedetto1983}. Next, using part (ii) and that $v_\varepsilon$ satisfies \eqref{eq_ajuda}, one can apply \cite[Theorem 1.1]{trudinger1967} to conclude $v_\varepsilon >0$ in $\mathbb{R}^N.$
\end{proof}
\begin{remark}\label{r_infpositivo}In the case where $q_0 = p$ in hypothesis \ref{f_brutal}, we may take $V_0 = \inf _{x \in \mathbb{R}^N}V(x)>0.$ Next, we modify the proof of Proposition \ref{p_linf} by choosing $b(x)=V(\varepsilon x) - V_0/2\geq V_0/2 $ and $h_0(x,t) = (Q_\varepsilon (x)-V_0/2)t^{p-1}$ in \eqref{eq_ajuda} and following the proof of Lemma \ref{l_chines} by noticing that $g(\varepsilon x,v_\varepsilon ) - (V_0/2) v_\varepsilon ^{p-1} \leq C_0v_\varepsilon ^{q-1} + v_\varepsilon ^{p^\ast -1}$ ($\epsilon_0 = V_0/2$ in \ref{f_brutal}).
\end{remark}
	\subsection{Uniform local convergence}\label{s_uniform}
	Now let $(\varepsilon _n) \subset (0, \varepsilon_0)$ with $\varepsilon_n \rightarrow 0$ and $(y_n) \subset \mathbb{R}^N.$ In order to prove our main result, we are going to consider the next equation
 \begin{equation}\label{Ptrans}
     -\Delta_pw+V(\varepsilon_n(x+y_n))w^{p-1}=g(\varepsilon_n(x+y_n),w), \ w>0,\  \text{in} \ \mathbb{R}^{N},
 \end{equation}
Solutions of \eqref{Ptrans} are given by critical points of the following $C^1$ functional
	\begin{equation*}
	I_{\varepsilon _n, y_n} (w) = \frac{1}{p} \| \nabla w \| _p ^p + \int _{\mathbb{R}^N}V(\varepsilon_n (x+y_n))|w|^p\dx - \int _{\mathbb{R}^N} G(\varepsilon_n (x + y_n), w )\dx,\ w \in W^{1,p}(\mathbb{R}^N).
	\end{equation*}
Clearly $w_n(x) := v_{\varepsilon_n}(x+y_{n}) $ is a critical point of $I_{\varepsilon _n, y_n}.$ In our argument to prove our main result, we study the behavior of $w_n,$ as $n\rightarrow \infty.$ In this direction, we first notice that \eqref{norma_est} implies the existence of $w_0$ such that $w_n \rightharpoonup w_0$ in $D^{1,p} (\mathbb{R}^N),$ up to a subsequence. Also, if we assume the existence of $V_\ast : = \lim_n V(\varepsilon _n y_n),$ up to a subsequence, then $w_0 \in W^{1,p}(\mathbb{R}^N),$ provided that $V_\ast>0$ (Fatou's lemma). 
	\begin{proposition}\label{p_convergcomp}
		Suppose either: the existence of $\lim (\varepsilon _n y_n )=x_{\Omega };$ Or $V\in L^\infty (\mathbb{R}^N)$ and it is uniformly continuous. Then, in a subsequence, $(w_n)$ and $(\nabla w_n) $ converges to $w_0$ and $\nabla w_0$ uniformly on compact sets of $\mathbb{R}^N$, respectively. Furthermore, $ w_0\in D^{1,p}(\mathbb{R}^N) \cap L^\infty (\mathbb{R}^N)\cap C^{1,\alpha }_{\loca}(\mathbb{R}^N)$ and it is a nonnegative solution of the equation
		\begin{equation}\label{Ptrans0}
		-\Delta_pw+V_\ast \hspace{0.03cm}   w^{p-1}=g_\ast (w) \  \text{in} \ \mathbb{R}^{N},\\
		\end{equation}
		where $g_\ast  (t) = \mathcal{X}  ( f(t) + t^{p^\ast -1}) + (1 - \mathcal{X} ) f_\ast (t)$ for some $\mathcal{X} \in \mathbb{R},$ and $g_\ast  (t)=0,$ for $t \leq 0.$ Solutions of \eqref{Ptrans0} are precisely critical points of the $C^1$ functional
		\begin{equation*}
		\mathcal{I}_{\ast}(w) = \frac{1}{p} \| \nabla w \| _p ^p + \int _{\mathbb{R}^N} V_\ast  |w|^p\dx - \int _{\mathbb{R}^N}G_\ast(w) \dx,\quad w \in W_{V_\ast}
		\end{equation*}
		where $G_\ast(w) = \int _0 ^t g_\ast (\tau )\dtau$ with $W_{V_\ast} = W^{1,p}(\mathbb{R}^N),$ if $V_\ast >0,$ and $W_{V_\ast} = D^{1,p}(\mathbb{R}^N),$ if $V_\ast =0.$ Moreover, if $w_0$ is nontrivial, then $w_0>0.$
	\end{proposition}
\begin{proof}
The compact embedding $D^{1,p}(\mathbb{R}^N) \hookrightarrow L ^s_{\loca}(\mathbb{R}^N),$ $1\leq s < p^\ast,$ implies $w_n(x) \rightarrow w_0(x),$ up to a subsequence. On the other hand, the regularity given by Proposition \ref{p_linf} guarantees that $w_0 \in L ^\infty (\mathbb{R}^N)$ with $\| w_0 \| _\infty \leq M_\ast.$ Given a ball $B,$ by \eqref{Ptrans}, we further apply the regularity results of \cite{dibenedetto1983} to ensure the existence of $C_0>0$ depending only upon $N,$ $p,$ $f,$ $B$ and $M_\ast$ such that
\begin{equation*}
	|\nabla w_{n}(x) - \nabla w_{n} (y) |\leq  C_0 |x- y |^\beta\quad \text{and} \quad\|\nabla w_n\|_{L^\infty (B)} \leq C_0,\quad  \text{for any}\quad x,\ y \in B\text{ and }n \in \mathbb{N}.
\end{equation*}
Consequently $(\| w_n \| _{C^{1,\beta }(\overline{B}) })$ is a bounded sequence and by the compact embedding $C^{1,\beta }(\overline{B})  \hookrightarrow C^1 (\overline{B})$ we conclude that $\| w_n - w_0\|_{C^1 (\overline{B})}\rightarrow 0,$ up to a subsequence. Using this convergence together with the fact that $C_0 ^\infty (\mathbb{R}^N)$ is dense in $X_\varepsilon,$ Proposition \ref{p_ge}--(i), \ref{f_brutal} and dominated convergence theorem, one can see that $w_0$ is a solution of \eqref{Ptrans0}, where $\mathcal{X}$ depends on the behavior of $(\varepsilon_ny_n)$ and on its limits in subsequences. The arguments used to prove Proposition \ref{p_linf} imply $w_0 \in C^{1,\alpha}_{\loca}(\mathbb{R}^N).$
\end{proof}
	
	\section{Proof of Theorem \ref{teo_doo}}\label{s_teodoo}
	In this section we apply the results of Section \ref{s_existence} assuming without loss of generality that $x_0 = 0$ ($V(0)>0$). In fact, one can just consider the translated problem in \eqref{P} replacing $V$ and $\Omega$ by $\hat{V}(x) = V(x+x_0)$ and $\hat{\Omega} = \Omega - x_0,$ respectively. Moreover, we repeat the same notation to indicate possible lesser values of $\varepsilon_0 \in (0,1)$. Next we relate the study of \eqref{PF} to the following limit equation, obtained by taking $\varepsilon \rightarrow 0,$
\begin{equation}\label{p_limit}
-\Delta_pu+V(0)u^{p-1}=f(u)+u^{p^\ast -1},\ u>0,\  \text{in} \ \mathbb{R}^{N},
\end{equation}
 whose solutions can be obtained by critical points of the functional 
	\begin{equation*}
	\mathcal{I}_0 (v) = \frac{1}{p} \int _{\mathbb{R}^N } |\nabla v|^p +V(0)|v|^p \dx - \int _{\mathbb{R}^N } F(v) + \frac{1}{p^\ast} (v^+)^{p^\ast}\dx,\ v \in W^{1,p}(\mathbb{R}^N).
	\end{equation*}
	We define the minimax level of $\mathcal{I}_0$ as $c(\mathcal{I}_0) := \inf _{\gamma \in \Gamma_{\mathcal{I}_0} } \sup _{t \geq 0} \mathcal{I}_0(\gamma (t)),$ where $\Gamma_{\mathcal{I}_0}$ is defined like \eqref{gamma}, replacing $I_\varepsilon$ by $\mathcal{I}_0,$ and $X_\varepsilon$ by $W^{1,p}(\mathbb{R}^N).$ It is well known the existence of a solution $u_0\in W^{1,p}(\mathbb{R}^N)\cap C^{1,\alpha }_{\loca}(\mathbb{R}^N)$ for \eqref{p_limit} such that
	\begin{equation}\label{minimax_est0}
	0<\mathcal{I}_0 (u_0) =\max_{t \geq 0}\mathcal{I}_0 (t u_0) = c(\mathcal{I}_0)< \frac{1}{N} \mathbb{S} ^{N/p},
	\end{equation}
	see \cite{wang-an-zhang2015}. Critical points of $\mathcal{I}_0$ satisfying \eqref{minimax_est0} can also be obtained by following similar arguments used in Section \ref{s_existence}. Our first step here is to obtain an estimate of the minimax level $c(I_\varepsilon)$ by means of $c(\mathcal{I}_0).$
	\begin{lemma}\label{l_est_j0}
		$	\limsup_{\varepsilon \rightarrow 0} c(I_{\varepsilon}) \leq c(\mathcal{I}_0).$ 
	\end{lemma}
	\begin{proof}
		Our proof is inspired by \cite[Lemma 5.4]{ambrosio2020}. Let $\psi \in C^\infty (\mathbb{R}^N : [0,1])$ such that $\psi(x)=1,$  if $|x|\leq r/2,$ and $\psi (x) =0,$ if $|x| \geq r,$ where $r>0$ is such that $B_r \subset \Omega.$ For a given $0<R\leq \varepsilon^{-1}$ we define $u _R(x) = \psi(x/R)u_0(x).$ Clearly $u_R \rightarrow u_0$ in $W^{1,p}(\mathbb{R}^N).$ Lemma \ref{l_geometry} guarantees the existence of $t_{\varepsilon,R} >0$ such that 
		\begin{equation*}
		c(I_\varepsilon) \leq \max_{t \geq 0} I_\varepsilon (t u_R) = I_\varepsilon (t_{\varepsilon,R} u_R).
		\end{equation*}
		Hence $\varphi ' (t_{\varepsilon,R}) = 0,$ where $\varphi (t )= I_\varepsilon (t u_R)$ and
		\begin{equation}\label{eq_te}
		\| u_R\|_{(\varepsilon)} ^p = \int _{B_{Rr}} \frac{f( t_{\varepsilon,R} u_R) + (t_{\varepsilon,R} u_R)^{p^\ast - 1}}{(t_{\varepsilon,R} u_R) ^{p-1}} u^p_R \dx.
		\end{equation}
		Let $\varepsilon_n \rightarrow 0$ and $R_k \rightarrow \infty.$ Consider in \eqref{eq_te}, $\varepsilon = \varepsilon_n$ and $R=R_k.$ By dominated convergence theorem, up to a subsequence, we deduce that \eqref{eq_te} implies $0<t_{\varepsilon_n,R_k} \rightarrow t_{R_k} <+ \infty.$ We take the limit as $\varepsilon=\varepsilon_n \rightarrow 0$ in \eqref{eq_te} to obtain
		\begin{equation}\label{eq_erre}
		\| u_R\|_{(0)} ^p = \int _{B_{Rr}} \frac{f( t_{R} u_R) + (t_{R} u_R)^{p^\ast - 1}}{(t_{R} u_R) ^{p-1}} u^p_R \dx,
		\end{equation}
		with $R=R_k.$ Similarly, taking $t_R = t_{R_k}$ in \eqref{eq_erre} we have $0<t_{R_k} \rightarrow t_0 < + \infty ,$ up to a subsequence. We now take the limit $R=R_k \rightarrow \infty$ in \eqref{eq_erre} and use that $\mathcal{I}'_0(u_0) \cdot u_0 = 0 $ to get
		\begin{equation*}
		\int _{\mathbb{R}^N} \frac{f(t_0 u_0) + (t_0 u_0 )^{p^\ast -1}}{(t_{0} u_0) ^{p-1}} u^p_0  \dx = \| u_0\|_{(0)} ^p =  \int _{\mathbb{R}^N} \frac{f( u_0)   + u_0 ^{p^{\ast -1} }}{ u_0 ^{p-1}} u^p_0  \dx.
		\end{equation*}
		Using \ref{f_increasing} we see that $t_0 = 1.$ Summing up, since $(\varepsilon_n)$ and $(R_k)$ are arbitrary, we have our conclusion:
		\begin{equation*}
		\limsup _{\varepsilon \rightarrow 0}c(I_\varepsilon) \leq \limsup _{R \rightarrow \infty }\left[ \limsup _{\varepsilon \rightarrow 0} I_\varepsilon (t_{\varepsilon,R} u_R)\right] =  \mathcal{I}_0(u_0) = c(\mathcal{I}_0).\qedhere
		\end{equation*}
	\end{proof}
		\begin{proposition}\label{p_estbaixo}
		$\lim _{\varepsilon \rightarrow 0} c(I_\varepsilon) = c (\mathcal{I}_0).$
	\end{proposition}
	\begin{proof}
		Take any $\varepsilon _n \rightarrow 0.$ Estimate \eqref{norma_est} yields $v_{\varepsilon_n} \rightharpoonup v_0$ in $D^{1,p}(\mathbb{R}^N),$ for some $v_0,$ up to a subsequence. Moreover, by the results of Section \ref{s_uniform} and Proposition \ref{p_convergcomp}, $v_0 \in W^{1,p}(\mathbb{R}^N)$ and it is a solution of \eqref{p_limit}. On the other hand, some arguments in the proof of Lemma \ref{l_geometry} together with hypothesis \ref{f_increasing} leads to 
		\begin{equation*}
		c(\mathcal{I}_0) \leq \max _{t \geq 0} \mathcal{I}_0 (t v_0) = \mathcal{I}_0 (v_0).
		\end{equation*}
		Nevertheless, by Proposition \ref{p_ge}--(iv), we can apply Fatou's lemma to get
		\begin{align*}
		\liminf_{n \rightarrow \infty} c(I_{\varepsilon_n}) = \liminf_{n \rightarrow \infty} I_{\varepsilon_n} ( v_{\varepsilon _n}) &= \liminf_{n \rightarrow \infty} \left[\frac{1}{p} \int _{\mathbb{R}^N} g(\varepsilon _n x, v_{\varepsilon _n} )v_{\varepsilon _n} - p G(\varepsilon_n x, v_{\varepsilon _n})   \dx\right]\\
		&\geq \frac{1}{p} \int _{\mathbb{R}^N} g( 0 , v_0 )v_0 - p G(0, v_0)   \dx = \mathcal{I}_0(v_0) \geq  c(\mathcal{I}_0).
		\end{align*}		
		Conclusion follows by the fact that $\varepsilon _n \rightarrow 0$ is arbitrary and Lemma \ref{l_est_j0}.
	\end{proof}
We point out that Proposition \ref{p_estbaixo} implies the existence of $\varepsilon_0 \in (0,1)$ and $C>0$ such that $\| v_\varepsilon \| _{(\varepsilon)} > C>0,$ for any $\varepsilon \in (0, \varepsilon _0).$ In fact, assuming the existence of $\varepsilon_n \rightarrow 0$ with $\| v_{\varepsilon _n}\| _{(\varepsilon _n)}=o_n(1),$ we have by Proposition \ref{p_ge}--(iv), $o_n(1) = \| v_{\varepsilon _n} \|^p_{(\varepsilon _n)} \geq p \int _{\mathbb{R}^N} G(\varepsilon_n x , v_{\varepsilon _n})\dx,$ leading to the contradiction with \eqref{minimax_est0}: $c(I_{\varepsilon_n}) = I ( v_{\varepsilon _n}) \rightarrow 0.$
	\begin{lemma}\label{l_lionzinho}
		There are $R_0,$ $\beta _0 >0$ and $(\hat{y}_\varepsilon)_{0<\varepsilon < \varepsilon _0} \subset \mathbb{R}^N$ such that
		\begin{equation*}
		\int _{B_{R_0} (\hat{y}_\varepsilon)} v_\varepsilon ^p \dx \geq \beta _0 > 0,\ \forall \, \varepsilon \in (0, \varepsilon_0).
		\end{equation*}
	\end{lemma}
	\begin{proof}
		Let us argue by a contradiction argument and suppose the existence of $\varepsilon_n \rightarrow 0$ such that
		\begin{equation*}
		\lim _{n \rightarrow \infty}  \sup _{x \in \mathbb{R}^N} \int _{B_R(x)} v_{\varepsilon_n } ^p  \dx =0,\ \forall \, R>0.
		\end{equation*}
		By Lion's compactness lemma \cite[Lemma I.1]{lionslocalcompII}, we have $v_{\varepsilon _n} \rightarrow 0$ in $L ^\theta (\mathbb{R}^N)$ for any $p < \theta <p^\ast.$ Therefore, by \ref{f_brutal},
		\begin{equation*}
		\int _{\mathbb{R}^N} f(v_{\varepsilon _n})v_{\varepsilon _n}\dx = \int _{\mathbb{R}^N} F(v_{\varepsilon _n}) \dx = o_n (1).
		\end{equation*}
		Using this convergence and Proposition \ref{p_ge}--(i) we obtain the following estimate
		\begin{equation}\label{estesp_um}
		\int _{\mathbb{R}^N}   G(\varepsilon _n x, v_{\varepsilon _n})\dx \leq \frac{1}{p^\ast } \int _{Y_n } v_{\varepsilon _n} ^{p^\ast } \dx + \frac{V_0}{p \kappa_0}\int _{Y^c_n} v_{\varepsilon _n} ^p \dx + o_n(1),
		\end{equation}
		where $Y_n = \Omega _{\varepsilon _n} \cup [v_{\varepsilon _n} \leq  a_0].$ On the other hand, from $I'_{\varepsilon _n} (v_{\varepsilon _n} ) \cdot v_{\varepsilon _n} = 0,$ we have
		\begin{equation}\label{estesp_dois}
		\| v_{\varepsilon _n} \| _{(\varepsilon _n)} ^p - \frac{V_0}{\kappa_0} \int _{Y_n ^c} v^p _{\varepsilon _n}\dx + o_n (1) = \int _{Y_n} v_{\varepsilon _n}^{p^\ast}\dx.
		\end{equation}
		Since $(\| v_{\varepsilon _n}\|^p_{(\varepsilon _n)})$ is bounded, by Sobolev inequality, clearly the right-hand side of \eqref{estesp_dois} is bounded, and we can take the limit
		\begin{equation*}
		l_0:= \lim_{n \rightarrow \infty} \left[ \| v_{\varepsilon _n} \| _{(\varepsilon _n)} ^p - \frac{V_0}{\kappa_0} \int _{Y_n ^c} v^p _{\varepsilon _n}\dx \right]=\lim_{n \rightarrow \infty}\int _{Y_n} v_{\varepsilon _n}^{p^\ast}\dx\geq 0,
		\end{equation*}
		up to a subsequence. Nevertheless, \eqref{usa} yields
		\begin{equation}\label{eq_faltou}
		\| v_{\varepsilon _n} \| _{(\varepsilon _n)} ^p - \frac{V_0}{\kappa_0} \int _{Y_n ^c} v^p _{\varepsilon _n}\dx \geq \| \nabla v_{\varepsilon _n} \|^p _p + V_0 \left(1 - \frac{1}{\kappa_0}\right) \int _{Y_n^c}v _{\varepsilon _n}^p \dx + \int _{Y_n } V(\varepsilon _n x ) v_{\varepsilon _n} ^p\dx,
		\end{equation}
		Consequently, if $ l_0= 0, $ then $\| \nabla v_{\varepsilon _n} \|_{p} = o_n (1)$ and we have $\|  v_{\varepsilon _n} \|_{p^\ast} = o_n(1).$ Thus
		\begin{equation*}
		\| v_{\varepsilon _n} \| _{\varepsilon _n}^p = \int _{\mathbb{R}^N} g(\varepsilon _n x, v_{\varepsilon _n} )v_{\varepsilon _n} \dx = o_n(1),
		\end{equation*}
		which is a contradiction with the fact that $\|v_\varepsilon\|_{(\varepsilon)}>C>0$ as above. On the other hand, by using Proposition \ref{p_estbaixo} and \eqref{estesp_um} in $I'(v_{\varepsilon _n}) \cdot v_{\varepsilon _n} =0,$ we obtain
		\begin{equation}\label{estesp_tres}
		\| v_{\varepsilon _n} \|_{\varepsilon _n}^p - \frac{V_0}{\kappa_0} \int _{Y_n ^c} v_{\varepsilon _n} ^p \dx - \frac{p}{p^\ast} \int _{ Y_n } v_{\varepsilon_n} ^{p ^\ast } \dx  \leq p c(\mathcal{I}_0)+o_n (1).
		\end{equation}
		Taking the limit as $n \rightarrow \infty$ in \eqref{estesp_tres} we have
		\begin{equation}\label{contrad1}
		\frac{1}{N}l_0 =\frac{1}{p}\left( 1- \frac{p}{p ^\ast} \right)l_0  \leq  c(\mathcal{I}_0).
		\end{equation}
		We now use \eqref{contrad1} to complete the proof. Indeed, using \eqref{usa} and \eqref{eq_faltou}, one have
		\begin{equation}\label{est_difi}
		\mathbb{S}  \left(  \int _{Y_n} v_{\varepsilon _n}^{p^\ast}\dx \right)^{p/p^\ast}\leq \mathbb{S}  \| v_{\varepsilon_n } \| _{p^\ast } ^p  \leq \| \nabla v_{\varepsilon _n} \| _p ^p \leq \|  v_{\varepsilon _n} \| _{(\varepsilon _n)} ^p - \frac{V_0}{\kappa_0} \int _{Y_n ^c} v^p _{\varepsilon _n}\dx.
		\end{equation}
		We take the limit as $n \rightarrow \infty$ in \eqref{est_difi} to get $\mathbb{S} l_0 ^{p/p^\ast }  \leq l_0.$ This estimate together with \eqref{contrad1} and the fact that $l_0>0$ lead to $ \mathbb{S}^{N/p} \leq l_0 \leq N c(\mathcal{I}_0),$ which is a contradiction with \eqref{minimax_est0}.
	\end{proof}
	\begin{lemma}\label{l_yaux}
		$\dist (\varepsilon \hat{y}_\varepsilon, \Omega ) \leq \varepsilon R_0$ and $(\varepsilon \hat{y}_\varepsilon)_{0<\varepsilon \leq \varepsilon _0}$ is bounded.
	\end{lemma}
	\begin{proof}
		Let us denote $A_{\Omega, \delta}  = \left\lbrace  x \in \mathbb{R}^N : \dist (x,\Omega ) \leq \delta \right\rbrace ,$ where $\delta >0. $ Consider $\psi \in C^\infty _0 (\mathbb{R}^N : [0,1])$ such that $\|\nabla \psi \|_\infty \leq C/\delta  $ with $\psi  = 1$ in $A_{\Omega, \delta} ^c$ and $\psi =0$ in $\Omega.$ Define $\psi _\varepsilon (x) = \psi (\varepsilon x ).$ Clearly $\psi _\varepsilon = 1$ in $\varepsilon^{-1}A_{\Omega, \delta}  ,$ $\psi _\varepsilon =0$ in $\Omega_\varepsilon $ and $\|\nabla \psi_\varepsilon \|_\infty \leq C \varepsilon/ \delta.$ Clearly $I'_{\varepsilon}(v_\varepsilon)\cdot (\psi _\varepsilon v_{\varepsilon}  ) = 0 $ implies
		\begin{align*}
		\int _{\mathbb{R}^N }  \left[  |\nabla v_\varepsilon|^p + \left(   V(\varepsilon x) - \frac{V_0}{\kappa_0}  \right)v^p_\varepsilon    \right]  \psi _\varepsilon\dx &= \int _{\mathbb{R}^N } \left( g(\varepsilon x, v_\varepsilon) v_\varepsilon  - \frac{V_0}{\kappa_0} v_\varepsilon ^p \right)\psi _\varepsilon\dx 	 \\
		&\qquad - \int _{\mathbb{R}^N} v_\varepsilon  |\nabla v_\varepsilon |^{p-2} \nabla v_\varepsilon \cdot \nabla \psi _\varepsilon  \dx.
		\end{align*}
		Nevertheless, since $\psi _\varepsilon = 0$ in $\Omega _\varepsilon,$ by Proposition \ref{p_ge}--(iii), we have
		\begin{equation*}
		\int _{\mathbb{R}^N } \left( g(\varepsilon x, v_\varepsilon) v_\varepsilon  - \frac{V_0}{\kappa_0} v_\varepsilon ^p \right)\psi _\varepsilon\dx \leq 0.
		\end{equation*}
		Also, by H\"{o}lder's inequality,
		\begin{equation*}
		- \int _{\mathbb{R}^N} v_\varepsilon  |\nabla v_\varepsilon |^{p-2} \nabla v_\varepsilon \cdot \nabla \psi _\varepsilon  \dx \leq \| \nabla \psi _\varepsilon \|_\infty \| v_\varepsilon \|_{L ^p (\Omega_\varepsilon ^c)}\| \nabla v _{\varepsilon}\|^{p-1}_{L ^p (\Omega_\varepsilon ^c)}.
		\end{equation*}
		Hence
		\begin{equation*}
		\int _{\Omega _\varepsilon^c}  \left[  |\nabla v_\varepsilon|^p + \left(   V(\varepsilon x) - \frac{V_0}{\kappa_0}  \right) v^p_\varepsilon   \right]  \psi _\varepsilon\dx \leq \|  \nabla \psi _\varepsilon \|_\infty \| \nabla v _{\varepsilon}\|^{p-1}_{L ^p (\Omega_\varepsilon ^c)} \left( \int _{\Omega ^c _\varepsilon} v_\varepsilon ^p  \dx \right)^{1/p}.
		\end{equation*}
		By \eqref{usa} it is easy to see that $V(\varepsilon x)  - V_0 / \kappa_0 \geq V_0 (1-1/\kappa_0)>0$ in $\Omega _\varepsilon^c,$ and by \eqref{norma_est}, $\| \nabla v _{\varepsilon}\|^{p-1}_{L ^p (\Omega_\varepsilon ^c)}$ is bounded. Hence
		\begin{equation}\label{peia}
		\int _{\Omega _\varepsilon ^c} \left(|\nabla v_\varepsilon |^p + v^p_\varepsilon  \right)\psi _\varepsilon \dx \leq C \frac{\varepsilon}{\delta } \left( \int _{\Omega ^c _\varepsilon} v_\varepsilon ^p  \dx \right)^{1/p}.
		\end{equation}
		Arguing as in \eqref{final1}, the integral on the right-hand side of \eqref{peia} is bounded for any $0  < \varepsilon \leq \varepsilon_0.$ Next, using \eqref{peia}, we are going to prove:
		\begin{equation}\label{gostei}
		B_{R_0} (\hat{y}_\varepsilon) \cap \{  x\in \mathbb{R}^N : \varepsilon x \in A_{\Omega, \delta}   \} \neq \emptyset,\quad \text{for all} \ 0<\varepsilon \leq \varepsilon_0.
		\end{equation}
		Indeed, suppose the existence of $0 < \varepsilon_\ast \leq \varepsilon_0$ such that $B_{R_0} (\hat{y}_{\varepsilon_\ast}) \cap (\varepsilon_\ast  ^{-1} A_{\Omega, \delta}  ) = \emptyset.$ In this case, we have $\psi _{\varepsilon_\ast } = 1$ in $B_{R_0} (\hat{y}_{\varepsilon_\ast})$ and by \eqref{peia} we get
		\begin{equation*}
		\int _{B_{R_0} (\hat{y}_{\varepsilon_\ast})} |\nabla v_{\varepsilon_\ast} |^p + v^p_{\varepsilon_\ast}   \dx \leq  C \frac{\varepsilon_\ast}{\delta } \left( \int _{\Omega ^c _{\varepsilon _\ast} } v_{\varepsilon _\ast} ^p  \dx \right)^{1/p} \rightarrow 0,\quad \text{as} \ \delta\rightarrow \infty.
		\end{equation*}
		This lead to a contradiction with Lemma \ref{l_lionzinho}. By \eqref{gostei} for each $0<\varepsilon \leq \varepsilon_0$ there is $z _\varepsilon$ such that $|\varepsilon z_\varepsilon - \varepsilon \hat{y}_\varepsilon| < \varepsilon R_0$ and $\dist(\varepsilon z_\varepsilon , \Omega ) \leq \delta .$ Hence $\dist(\varepsilon \hat{y}_\varepsilon , \Omega ) \leq \varepsilon R_0 +\delta,$ for all $\delta >0.$ If there is $(\varepsilon_n)\subset (0,\varepsilon_0)$ such that $|\varepsilon _n \hat{y}_{\varepsilon_n}| \rightarrow \infty ,$ then we would have following contradiction:
		\begin{equation*}
		\varepsilon_n R_0 \geq |\varepsilon_n \hat{y}_{\varepsilon_n} - \omega _n | \geq |\varepsilon _n \hat{y}_{\varepsilon_n}| - \sup_{\omega \in \Omega } |\omega| \rightarrow \infty, \quad  \text{as }n \rightarrow \infty,
		\end{equation*}
		where $|\varepsilon_n \hat{y}_{\varepsilon_n} - \omega _n | = \dist(\varepsilon_n \hat{y}_{\varepsilon_n}  , \Omega).$
	\end{proof}
	In what follows, we are going to consider the set $\mathcal{C}_\Omega = \{ x \in \Omega : V(x) \geq V(0) \}.$ Condition \ref{V_delpinofelmer} implies $\mathcal{C}_\Omega \neq \emptyset$ and $\partial \Omega \subset \overline{\mathcal{C}_\Omega }.$
	\begin{proposition}\label{p_brutal} There exist $R>0$ and $(y_\varepsilon)_{0<\varepsilon \leq \varepsilon _0} \subset \mathbb{R}^N$ such that $\varepsilon y_\varepsilon \in \mathcal{C}_{\Omega} $ and
		\begin{equation}\label{faltou}
		\int _{B_{R} (y_\varepsilon)} v_\varepsilon ^p \dx \geq \beta _0 > 0,\quad \text{for any } 0 <\varepsilon \leq \varepsilon_0.
		\end{equation}
	\end{proposition}
	\begin{proof}
	 Let $\hat{w}_\varepsilon \in \partial \Omega$ such that $\dist(\varepsilon \hat{y}_\varepsilon , \Omega) = |\varepsilon \hat{y}_\varepsilon - \hat{w} _\varepsilon |.$ Using \ref{V_delpinofelmer}, we have $V(\hat{w} _\varepsilon)>V(0)$ and by continuity one can get $\varrho _\varepsilon \in (0, \varepsilon \delta _\Omega ),$ with $\delta _\Omega = {\rm diam} (\mathcal{C}_\Omega),$ in a such way that $B_{\varrho _\varepsilon} (\hat{w}_\varepsilon )\cap \Omega \subset  \mathcal{C}_\Omega.$ Choosing $w_\varepsilon \in B_{\varrho _\varepsilon} (\hat{w}_\varepsilon )\cap \Omega ,$ by Lemma \ref{l_yaux}, it is clear that $| \varepsilon\hat{y}_\varepsilon- w _\varepsilon|\leq  \varepsilon (R_0+\delta_\Omega ).$ Hence $B_{R_0}(\hat{y}_\varepsilon) \subset B_{2R_0+\delta _\Omega} (\varepsilon ^{-1} w_\varepsilon),$ which allow us to take $R=2R_0+\delta_\Omega $ and $y _\varepsilon = \varepsilon ^{-1} w_\varepsilon$ to use Lemma \ref{l_lionzinho} and obtain \eqref{faltou}.
	\end{proof}
Taking any $\varepsilon_n \rightarrow 0$ and $y_n = y_{\varepsilon _n}$ given by Proposition \ref{p_brutal}, we see that $\varepsilon_n y_n \rightarrow x_\Omega \in \overline{\mathcal{C}}_\Omega,$ up to a subsequence. Now we consider \eqref{Ptrans}, $w_n = v_{\varepsilon _n} (\cdot + y_n) \rightharpoonup w_0$ in $D^{1,p}(\mathbb{R}^N),$ and the related limit problem
\begin{equation*}
	\left\{
	\begin{aligned}
		&-\Delta_pw+V(x _\Omega )  w^{p-1}=g_\Omega  (w) \  \text{ in } \ \mathbb{R}^{N},\\
		&w\in D^{1,p}(\mathbb{R}^N)\cap C^{1,\alpha}_{\loca}(\mathbb{R}^N) ,\ w \geq 0 \  \text{ in } \ \mathbb{R}^{N},
	\end{aligned}
	\right.
\end{equation*}
together with its solution $w_0\in W^{1,p}(\mathbb{R}^N)$ given by Proposition \ref{p_convergcomp}, where $V_\ast = V(x_\Omega),$ 
\begin{equation*}
    g_\ast (t)= g_\Omega  (t) = \mathcal{X}_\Omega(x _\Omega )  ( f(t) + t^{p^\ast -1}) + (1 - \mathcal{X}_\Omega(x _\Omega ) ) f_\ast (t),
\end{equation*}
$\mathcal{X}_\ast = \mathcal{X}_\Omega (x _\Omega)$ and $\mathcal{I}_\ast = \mathcal{I}_{x_\Omega }.$ Proposition \ref{p_brutal} implies
\begin{equation*}
0 < \beta _0\leq \int _{B_{R} (y_n)} v_{\varepsilon _n} ^p \dx = \int _{B_{R}} v_{\varepsilon _n}^p(\cdot + y_n)  \dx = \int _{B_{R}} w_n ^p\dx.
\end{equation*}
Consequently, using compact Sobolev embeddings, $w_0 >0.$
\begin{proposition}\label{p_vzao}
	$\lim_{n \rightarrow \infty} V(\varepsilon _n y_n)  = V(0)= V(x_\Omega)$ and $x_\Omega \in \Omega.$
\end{proposition}
\begin{proof}
We have $c(\mathcal{I}_0)  \leq \max_{t \geq 0 } \mathcal{I}_0(w_0) = \mathcal{I}_0(t_0 w_0),$ for some $t_0>0$ (see \cite{wang-an-zhang2015}). Next we use Proposition \ref{p_ge}--(i) and $V(0) \leq V(x_\Omega )$ to get that $\mathcal{I}_0(t_0 w_0) \leq \mathcal{I}_{x_\Omega} (t_0 w_0).$ Thus
	\begin{equation}\label{prova_lema}
	c(\mathcal{I}_0)  \leq \mathcal{I}_{x_\Omega} (t_0 w_0) \leq \max _{t \geq 0} \mathcal{I}_{x_\Omega} (t w_0) = \mathcal{I}_{x_\Omega}  (w_0),
	\end{equation}
	where the last equality follows by the last statement of Lemma \ref{l_geometry}. By Proposition \ref{p_ge}--(iv) we can use Fatou's lemma to do a change of variables and obtain
	\begin{align*}
	c(\mathcal{I}_0) \leq \mathcal{I}_{x_\Omega}  (w_0) &= \frac{1}{p} \int _{\mathbb{R}^n} \mathcal{G}(x_\Omega , w_0) \dx \\
	& \leq \frac{1}{p}\liminf _{n \rightarrow \infty } \int _{ \mathbb{R}^N} \mathcal{G}(\varepsilon_n (x+y_n ), w_n)  \dx =\frac{1}{p}\liminf _{n \rightarrow \infty } \int _{\mathbb{R}^N} \mathcal{G}(\varepsilon_n x, v_{\varepsilon_n})  \dx.
	\end{align*}
	Consequently, by Proposition \ref{p_estbaixo},
	\begin{equation}\label{prova_lema2}
	c(\mathcal{I}_0)  \leq\frac{1}{p}\liminf _{n \rightarrow \infty } \int _{ \mathbb{R}^N } \mathcal{G}(\varepsilon_n x, v_{\varepsilon_n})  \dx = \liminf _{n \rightarrow \infty } \left(I_{\varepsilon _n}(v_{\varepsilon_n}) - \frac{1}{p }I_{\varepsilon _n}'(v_{\varepsilon_n})\cdot v_{\varepsilon_n} \right) =  c(\mathcal{I}_0).
	\end{equation}
	We know that $V(0) \leq V(x_\Omega).$ If $V(0) < V(x_\Omega)$ then $\mathcal{I}_0(t_0 w_0) < \mathcal{I}_{x_\Omega} (t_0 w_0),$ which from \eqref{prova_lema} would lead to $c(\mathcal{I}_0)  < \mathcal{I}_{x_\Omega}  (w_0),$ a contradiction with \eqref{prova_lema2}. Hence $V(x_\Omega )=V(0) $ and by \ref{V_delpinofelmer} we have $x_\Omega \in \Omega.$\end{proof}
Next we consider the sets
\begin{multline*}
\mathcal{P}_n =[v_{\varepsilon _n} \geq a_0] \cap \left\lbrace x \in \mathbb{R}^N : \varepsilon_n x \in \Omega ^c \right\rbrace,\ \\ \mathcal{Q}_n =[w_{n} \geq a_0]\cap \left\lbrace  x \in \mathbb{R}^N :\varepsilon_n x + \varepsilon_n y_n \in \Omega ^c \right\rbrace  \text{ and }Q_\ast = [w_{0} \leq a_0]\cup \left\lbrace  x \in \mathbb{R}^N :x_\Omega \in \Omega \right\rbrace.
\end{multline*}
We prove that $|\mathcal{Q}_n| \rightarrow 0.$ Clearly $\mathcal{Q}_n = \mathcal{P}_n - y_n$ and so $|\mathcal{Q}_n| = |\mathcal{P}_n|.$ We first establish a convergence result involving $c(I_{\varepsilon}).$
	\begin{lemma}\label{l_melhorado}
		$	c(\mathcal{I}_0) = \liminf _{n \rightarrow \infty } \int _{ \mathcal{P}^c_n} p^{-1} \mathcal{G}(\varepsilon_n x, v_{\varepsilon_n})  \dx = \liminf _{n \rightarrow \infty } \int _{\mathbb{R}^N} p^{-1} \mathcal{G}(\varepsilon_n x, v_{\varepsilon_n})  \dx .$ 
	\end{lemma}
	\begin{proof}
		By Proposition \ref{p_vzao}, $Q_\ast=\mathbb{R}^N$ and we can write $\mathcal{G}(x_\Omega , w_0)  = \mathcal{G}(x_\Omega , w_0)  \mathcal{X}_{Q_\ast}.$ Since $\mathcal{X}_{Q_n^c} \rightarrow \mathcal{X}_{Q_\ast}$ in $\mathbb{R}^N,$ Fatou's lemma yields
		\begin{equation*}
		\int _{\mathbb{R}^n} \mathcal{G}(x_\Omega , w_0) \dx = \int _{Q_\ast } \mathcal{G}(x_\Omega , w_0) \dx \leq \liminf _{n \rightarrow \infty } \int _{ \mathcal{Q}^c_n} \mathcal{G}(\varepsilon_n (x+y_n ), w_n)  \dx=\liminf _{n \rightarrow \infty } \int _{ \mathcal{P}^c_n} \mathcal{G}(\varepsilon_n x, v_{\varepsilon_n})  \dx,
		\end{equation*}
		where in the last equality we used the change of variables theorem. Thus we can follow the proof of Proposition \ref{p_vzao} to get
	\begin{align}
	 c(\mathcal{I}_0) \leq \mathcal{I}_{x_\Omega}  (w_0) & \leq \frac{1}{p}\liminf _{n \rightarrow \infty } \int _{ \mathcal{Q}^c_n} \mathcal{G}(\varepsilon_n (x+y_n ), w_n)  \dx \nonumber \\
	 &=\frac{1}{p}\liminf _{n \rightarrow \infty } \int _{ \mathcal{P}^c_n} \mathcal{G}(\varepsilon_n x, v_{\varepsilon_n})  \dx\leq c(\mathcal{I}_0). \label{newproof}
	\end{align}
	This ends the proof.
	 \end{proof}
	\begin{lemma}\label{l_bemtec}
		$\lim _{n \rightarrow \infty }|\mathcal{Q}_n| =\lim _{n \rightarrow \infty }|\mathcal{P}_n| = 0,$ up to a subsequence.
	\end{lemma}
	\begin{proof}
		It is clear that Lemma \ref{l_melhorado} implies
		\begin{equation*}
		c(\mathcal{I}_0) = \frac{1}{p}\lim _{n \rightarrow \infty } \int _{ \mathbb{R}^N} \mathcal{G}(\varepsilon_n x, v_n)  \dx= \frac{1}{p}\lim _{n \rightarrow \infty } \int _{ \mathcal{P}^c_n} \mathcal{G}(\varepsilon_n x, v_n)  \dx,
		\end{equation*}
		up to a subsequence. Hence, in this subsequence,
		\begin{equation*}
		\frac{1}{p}\lim _{n \rightarrow \infty }\left[   \left(  \left( f(a_0)a_0  - p F(a_0) \right)      + \left(1-\frac{p}{p^\ast } \right)a_0^{p^\ast }       \right) |\mathcal{P}_n|     \right] =	\frac{1}{p}\lim _{n \rightarrow \infty } \int _{\mathcal{P}_n} \mathcal{G}(\varepsilon_n x, v_n)  \dx = 0. \qedhere
		\end{equation*}
	\end{proof}
	\begin{lemma}\label{l_critico}
		$w_n \mathcal{X}_{\mathcal{Q}^c_n} \rightarrow w_0$ in $L^{p^\ast }(\mathbb{R}^N).$
	\end{lemma}
	\begin{proof}
		In what follows, let us denote
		\begin{equation*}
		\mathcal{H}(t) = f(t)t - p F(t)\quad \text{and}\quad \mathcal{H}_\ast(t) = \left(1 - \frac{p}{p^\ast} \right)t^{p^\ast },\ t \geq 0.
		\end{equation*}
		Since $\mathcal{I}'_{x_\Omega}  (w_0)=0$ and $x_\Omega \in \Omega$ we have by Proposition \ref{p_vzao}, Lemma \ref{l_melhorado} and \eqref{newproof},
		\begin{align}
		c(\mathcal{I}_0) = \mathcal{I}_{x_\Omega}  (w_0)=\mathcal{I}_0(w_0) &=	\frac{1}{p} \left[ \int _{\mathbb{R}^N} \mathcal{H}(w_0)\dx + \int _{\mathbb{R}^N} \mathcal{H}_\ast (w_0) \dx \right]\nonumber\\
		&=  \frac{1}{p}\lim _{n \rightarrow \infty } \int _{\mathcal{Q}^c_n } \mathcal{G}(\varepsilon_n (x+y_n ), w_n)  \dx.\label{vem_um}
		\end{align}
		Now using the definition of $g$ and $\mathcal{Q}^c_n$ we get
		\begin{equation}
		\frac{1}{p}	\lim _{n \rightarrow \infty } \left[\int _{ \mathcal{Q}^c_n} \mathcal{H}(w_n) \dx + \int _{\mathcal{Q}^c_n} \mathcal{H}_\ast (w_n)\dx\right] = \frac{1}{p}\lim _{n \rightarrow \infty } \int _{ \mathcal{Q}^c_n} \mathcal{G}(\varepsilon_n (x+y_n ), w_n)  \dx.\label{vem_dois}
		\end{equation}
		On the other hand, by Proposition \ref{p_linf} and Lemma \ref{l_bemtec},
		\begin{equation*}
		\int _{\mathcal{Q}_n} w_n^{p^\ast } \dx,\	\int _{\mathcal{Q}_n} f(w_n)w_n\dx,\ \int _{\mathcal{Q}_n} F(w_n)\dx \leq C'|\mathcal{Q}_n|\rightarrow 0,\text{ as }n \rightarrow \infty,
		\end{equation*}
		for some $C'= C'(M_\ast)>0.$ Therefore,
		\begin{equation}\label{vem_tres}
		\lim _{n \rightarrow \infty }\int _{\mathcal{Q}^c_n}  \mathcal{H} (w_n) \dx =  \lim _{n \rightarrow \infty }\int _{\mathbb{R}^N }  \mathcal{H} (w_n) \dx\quad\text{and}\quad \lim _{n \rightarrow \infty }\int _{\mathcal{Q}^c_n} \mathcal{H}_\ast(w_n)   \dx = \lim _{n \rightarrow \infty } \int _{\mathbb{R}^N }  \mathcal{H} _\ast (w_n) \dx .
		\end{equation}
		Fatou's lemma yields,
		\begin{equation}\label{l_fatou}
		\int_{\mathbb{R}^N} \mathcal{H} (w_0)\dx  \leq \liminf _{n \rightarrow \infty}\int_{\mathbb{R}^N}\mathcal{H} (w_n) \dx\quad \text{and}\quad \int_{\mathbb{R}^N}\mathcal{H}_\ast  (w_0) \dx  \leq \liminf _{n \rightarrow \infty} \int_{\mathbb{R}^N} \mathcal{H}_\ast (w_n) \dx .
		\end{equation}
		If one of the inequalities in \eqref{l_fatou} is strict, then up to a subsequence, by \eqref{vem_um}, \eqref{vem_dois} and \eqref{vem_tres} we would have the following contradiction:
		\begin{align*}
		\int_{\mathbb{R}^N} \mathcal{H} (w_0)\dx + \int_{\mathbb{R}^N}\mathcal{H}_\ast  (w_0) \dx 
		&< \liminf _{n \rightarrow \infty}\int_{\mathbb{R}^N}\mathcal{H} (w_n) \dx + \liminf _{n \rightarrow \infty} \int_{\mathbb{R}^N} \mathcal{H}_\ast (w_n) \dx\\
		& \leq \liminf _{n \rightarrow \infty}\left[   \int_{\mathbb{R}^N}\mathcal{H} (w_n) \dx +  \int_{\mathbb{R}^N} \mathcal{H}_\ast (w_n) \dx     \right]\\
		& = \int_{\mathbb{R}^N} \mathcal{H} (w_0)\dx + \int_{\mathbb{R}^N}\mathcal{H}_\ast  (w_0) \dx. 
		\end{align*}
		Consequently, by \eqref{vem_tres}, we conclude
		\begin{equation*}
		\lim _{n \rightarrow \infty }\int _{\mathcal{Q}^c_n}  \mathcal{H} (w_n) \dx = \int_{\mathbb{R}^N} \mathcal{H} (w_0)\dx \quad \text{and}\quad \lim _{n \rightarrow \infty } \int _{\mathcal{Q}^c_n }  \mathcal{H} _\ast (w_n) \dx = \int_{\mathbb{R}^N}\mathcal{H}_\ast  (w_0) \dx.
		\end{equation*}
		In particular, $\| w_n \mathcal{X}_{\mathcal{Q}^c_n } \|_{p^\ast} \rightarrow \| w_0 \|_{p^\ast} .$ Since $w_n \mathcal{X}_{\mathcal{Q}^c_n } = w_n - w_n \mathcal{X}_{\mathcal{Q}_n } \rightharpoonup w_0$ in $L^{p^\ast} (\mathbb{R}^N),$ conclusion follows by the fact that $\| \cdot \| _{p^\ast}$ is strictly convex.
	\end{proof}
\begin{proposition}\label{p_convergenciaast}
$w_n \rightarrow w_0$ in $L^{p ^\ast}(\mathbb{R}^N),$ up to a subsequence.
\end{proposition}
\begin{proof}
Following Lemmas \ref{l_bemtec} and \ref{l_critico}, we have $w_n= w_n \mathcal{X}_{\mathcal{Q}^c_n }  + w_n \mathcal{X}_{\mathcal{Q}_n } \rightarrow w_0$ in $L^{p^\ast}(\mathbb{R}^N).$
\end{proof}
	We use the local behavior of solutions for quasilinear problems to prove the uniform decay of $(w_n).$ 
	\begin{lemmaletter}\label{l_regularity}\cite[Theorem 1.3]{trudinger1967}
		Let $w\in W^{1,p}_{\loca}(U)$ be a weak subsolution of the equation $-\Delta _p w = \mathcal{B}(x,w)$ in $ K=B_{3\varrho}(\xi_0) \subset U,$ where $U$ is a open domain in $\mathbb{R}^N$ and $\mathcal{B} :U \times \mathbb{R} \rightarrow \mathbb{R}$ is a continuous function. More precisely, $w$ satisfies
		\begin{equation*}
		\int_K |\nabla w|^{p-2}\nabla w \cdot \nabla \phi \dx \leq \int_K \mathcal{B}(x,w)\phi \dx,\quad \forall \ \phi \in W^{1,p} _0 (K),\ \phi \geq0.	
		\end{equation*}
		Assume the existence of $\mathcal{C}_\ast>0$ such that $|\mathcal{B} (x,t)| \leq \mathcal{C}_\ast |t|^{p-1}$ in $U\times [-M,M],$ for some $M>0.$ If $0 \leq w <M$ in $K,$ then 
		\begin{equation}\label{regularity}
		\max _{B_\varrho(\xi_0)} w \leq \mathcal{C}_0 \| w \| _{L^p (B_{2 \varrho}(\xi_0))},
		\end{equation}
		where $\mathcal{C}_0 = \mathcal{C}_0 (\varrho, p,\mathcal{C}_\ast, M )>0$ and $\max w$ mean the essential maximum of $w.$
	\end{lemmaletter}
	\begin{proposition}\label{p_decay}
		$w_\varepsilon (x)=v_\varepsilon (x+y_\varepsilon) \rightarrow 0$ as $|x| \rightarrow \infty,$ uniformly on $(0, \varepsilon_0),$ where $(\varepsilon y_\varepsilon) \subset \mathcal{C}_\Omega$ is given by Proposition \ref{p_brutal}.
	\end{proposition}
	\begin{proof}
		Proposition \ref{p_convergenciaast} yields:  for $\epsilon_0 >0$ there are $R_0>0$ and $n_0 \in \mathbb{N}$ such that
		\begin{equation}\label{vitali}
		\int _{B^c_R} w_n ^{p^\ast }\dx <\epsilon_0, \quad \forall \, R\geq R_0,\ \forall \, n \geq n_0.
		\end{equation}
		Next we apply Lemma \ref{l_regularity} by taking $\mathcal{B}(x,t) = b_0(x)|t|^{p-2}t,$ with $b_0(x) = g(\varepsilon_n(x+y_n),w_n)w_n^{1-p}.$ By Proposition \ref{p_ge}--(i), it is clear that
		\begin{equation*}
		|b_0(x)|  \leq M_\ast ^{q_0 - p} + C' M_\ast ^{q-p} + M_\ast ^{p^\ast -p},\ \forall \, x \in \mathbb{R}^N,
		\end{equation*}
		for some $C'>0,$ where $M_\ast$ is given by Proposition \ref{p_linf}. Thus $|\mathcal{B}(x,t)| \leq C_\ast |t|^{p-1}$ in $\mathbb{R}^N \times \mathbb{R},$ where $C_\ast>0$ does not depend on $n.$ By \eqref{regularity} and H\"{o}lder's inequality,
		\begin{equation}\label{regularity2}
		\max _{B_{\varrho}(\xi_0)} w_n \leq C_0 \| w_n \| _{L^p (B_{2\varrho}(\xi_0 ))}\leq  C_0 \omega _N (2 \varrho)^{1/N} \left( \int _{B_{2\varrho} (\xi_0)}   w^{p^\ast}_n \dx\right)^{1/p^\ast},
		\end{equation}
		where $\omega _N $ is the volume of the unit ball in $\mathbb{R}^N$ and $C_0=C_0(\varrho,p,q_0,q,M_\ast)>0.$ Consequently, for a fixed $\varrho,$ moving away $\xi _0$ in \eqref{regularity2} and taking into account \eqref{vitali}, the uniform vanishing of $(w_n)$ is implied. The proof follows because $\varepsilon_n \rightarrow 0$ is arbitrary.
	\end{proof}
	\begin{proof}[Proof of Theorem \ref{teo_doo} completed]
		We divide our proof in some steps.
		
		(i) By Proposition \ref{p_decay}, we know that there is $R_1 >0$ such that $w_\varepsilon (x) \leq a_0$ for $|x|\geq R_1,$ for any $0<\varepsilon<\varepsilon_0.$ Thus $v_\varepsilon (x) = w_\varepsilon (x-y_\varepsilon) <a_0,$ whenever $x \in B^c_{R_1}(y_\varepsilon).$ On the other hand, by Proposition \ref{p_vzao}, we have $x_\Omega \in \Omega$ and the existence of $\delta,$ $\varepsilon_0 \in (0,1)$ small enough such that $B_\delta (\varepsilon y_\varepsilon)\subset \Omega ,$ for $\varepsilon \in (0,\varepsilon_0).$ In particular, $B_{\varepsilon^{-1} \delta } (y_\varepsilon)\subset \Omega _\varepsilon$ and taking a smaller $\varepsilon_0$ if necessary, we have $B_{R_1}(y_\varepsilon) \subset B_{\varepsilon^{-1} \delta } (y_\varepsilon)  \subset \Omega _\varepsilon.$ Hence $g(\varepsilon x,v_\varepsilon) = f(v_\varepsilon) + v_\varepsilon^{p^\ast -1}$ and $v_\varepsilon$ is a positive solution of \eqref{P0}. The function $u_\varepsilon (x) := v_\varepsilon (x/\varepsilon)$ is a solution of \eqref{P}.
		
		(ii) Applying Proposition \ref{p_decay} we obtain $z^\ast _\varepsilon \in \mathbb{R}^N$ such that
  \begin{equation*}
      v_\varepsilon (z^\ast_\varepsilon + y_\varepsilon)=w_\varepsilon (z^\ast_\varepsilon) =\max _{\overline{B}_{\varrho_\varepsilon}} w_\varepsilon= \max _{\mathbb{R}^N}w_\varepsilon = \max _{\mathbb{R}^N}v_\varepsilon,
  \end{equation*}
  for some $\varrho_\varepsilon>0.$ In particular, $z_\varepsilon := \varepsilon (z^\ast_\varepsilon + y_\varepsilon)$ is a maximum point of $u_\varepsilon.$ On the other hand, we known that for any $\epsilon_0>0$ there is $R_0>0$ such that $w_\varepsilon (x) < \epsilon _0,$ when $|x|\geq R_0,$ for all $\varepsilon \in (0,\varepsilon_0).$ If there is $(z^\ast _{\varepsilon_n})$ such that $|z^\ast _{\varepsilon _n}| \rightarrow \infty,$ then $\|w_{\varepsilon_n} \|_\infty = w_{\varepsilon_n}(z^\ast _{\varepsilon_n})<\epsilon_0,$ for $n$ large enough. In particular, $\|w_{\varepsilon_n} \|_\infty < \epsilon_0$ for any $\epsilon_0>0,$ and we have the contradiction: $w_{\varepsilon_n}=0.$ Thus, taking smaller $\varepsilon_0,$ there is $R_2>0$ such $z^\ast _\varepsilon \in B_{R_2} \subset \Omega _\varepsilon$ for all $\varepsilon \in (0,\varepsilon_0).$ Thus, by Proposition \ref{p_vzao}, it is clear that $\lim _{\varepsilon \rightarrow 0}V(z_\varepsilon) = V(x_\Omega) = V(0).$ Moreover, by Proposition \ref{p_linf}, $u_\varepsilon (z_\varepsilon) = \| v_\varepsilon \|_\infty \leq M_\ast ,$ for any $\varepsilon \in (0,\varepsilon_0).$
		
		(iii) Now we prove the exponential decay of solutions. Take $\hat{w}_\varepsilon (x) := v_\varepsilon (x+ \hat{z}_\varepsilon),$ where $\hat{z}_\varepsilon = z^\ast_\varepsilon + y_\varepsilon$ is a maximum point of $v_\varepsilon$ as above. Since $\hat{w}_\varepsilon (x)\rightarrow 0,$ as $|x|\rightarrow \infty,$ uniformly in $\varepsilon \in (0,\varepsilon_0),$ there is $R_3>0$ with
		\begin{equation}\label{maisum}
		\frac{f(\hat{w}_\varepsilon) + \hat{w}_\varepsilon^{p^\ast -1}}{\hat{w}_\varepsilon^{p-1}}\leq \frac{1}{2}V_0 \quad \text{in } B_{R_3}^c,\ \forall \, \varepsilon \in (0,\varepsilon_0).
		\end{equation}
		Let any $R'_3>0$ such that $B_{R'_3}\supset \Omega $ and choose $R_\varepsilon > \max\{R_3, \varepsilon^{-1}(R'_3 + \sup_{\varepsilon \in (0,\varepsilon_0) }( \varepsilon|\hat{z}_\varepsilon|)) \}.$ Clearly, if $|x|\geq R_\varepsilon,$ then $|\varepsilon(x + \hat{z}_\varepsilon)| \geq R'_3.$ Thus $V(\varepsilon (x + \hat{z}_\varepsilon))\geq V_0$ in $B^c_{R_\varepsilon},$ and from \eqref{maisum} we have
		\begin{equation}\label{eita1}
		-\Delta _p \hat{w}_\varepsilon + \frac{V_0}{2} \hat{w}_ \varepsilon ^{p-1} \leq 0 \quad \text{in } B_{R_\varepsilon}^c,
		\end{equation}
		in the weak sense, more precisely:
		\begin{equation*}
		\int_{\mathbb{R}^N}|\nabla \hat{w}_\varepsilon|^{p-2} \nabla \hat{w}_\varepsilon \cdot \nabla \varphi +\frac{V_0}{2} \hat{w}_ \varepsilon ^{p-1} \varphi \dx \leq 0, \quad \forall \, \varphi \in W^{1,p}(B^c_{R_\varepsilon}),\ \varphi \geq 0.
		\end{equation*}
		On the other hand, choose $0<(p-1)\alpha^p< V_0/2$ and take $M_\varepsilon >0$ such that $\hat{w}_\varepsilon (x) \leq M_\varepsilon \exp (- \alpha |x|)$ for all $|x|=R_\varepsilon.$ Next we define $\psi_\varepsilon  (x) = M_\varepsilon \exp( - \alpha |x|)$ and use the radial form of the $p$-Laplacian operator to see that
		\begin{equation}\label{eita2}
		-\Delta_p \psi_\varepsilon  +\frac{V_0}{2} \psi _\varepsilon ^{p-1} = \left(\frac{N-1}{|x|}\alpha^{p-1}+ \frac{V_0}{2} - (p-1)\alpha ^p   \right) \psi _\varepsilon ^{p-1}>0\quad \text{in } \mathbb{R}^N,
		\end{equation}
		in the weak sense.
		Following we consider the function
		\begin{equation*}
		\eta _\varepsilon (x) = 
		\left\{
		\begin{aligned}
		&(\hat{w}_\varepsilon - \psi_\varepsilon)^+(x),& \ \text{if }x \in B_{R_\varepsilon} ^c,&\\
		&0, & \ \text{if }x \in B_{R_\varepsilon}.&
		\end{aligned}
		\right.
		\end{equation*}
		Taking $\eta _\varepsilon \in W^{1,p}(B^c_{R_\varepsilon})$ as a test function in \eqref{eita1} and \eqref{eita2} we have
		\begin{align}
		0 & \geq \int _{\mathbb{R}^N} \left( |\nabla \hat{w}_\varepsilon |^{p-2}\nabla \hat{w}_\varepsilon - |\nabla \psi_\varepsilon | ^{p-1}\nabla \psi_\varepsilon \right) \cdot \nabla \eta_\varepsilon +\frac{V_0}{ 2} (\hat{w}_\varepsilon^{p-1} - \psi _\varepsilon^{p-1} )\eta_\varepsilon\dx \nonumber  \\ 
		&\geq \frac{V_0}{2} \int _{A_\varepsilon} (\hat{w}_\varepsilon^{p-1} - \psi_\varepsilon ^{p-1} )(\hat{w}_\varepsilon-\psi_\varepsilon )\dx \geq 0,\label{in_simon}
		\end{align}
		where $A_\varepsilon = B_{R_\varepsilon} ^c \cap [\hat{w}_\varepsilon > \psi _\varepsilon]$ and in both inequalities we used an inequality due to J. Simon \cite[Lemma 2.1]{simon}. Precisely, there exists $C_0>0$ such that
		\begin{equation*}
		\left\langle |x|^{p-2}x - |y|^{p-2}y,x-y \right\rangle \geq \left\lbrace 
		\begin{aligned}
		&C_0|x-y|^p, &&\text{ if }p \geq 2,\\
		&C_0|x-y|^2 (|x| + |y|)^{2-p},&&\text{ if }1<p<2,
		\end{aligned}
		\right.
		\end{equation*}
		where $x,y \in \mathbb{R}^N$ and $\left\langle \cdot, \cdot  \right\rangle $ denotes the standard scalar product in $\mathbb{R}^N.$ Inequality \eqref{in_simon} implies $|A_\varepsilon|=0,$ which by continuity leads to $B_{R_\varepsilon}^c \cap [\hat{w}_\varepsilon \leq \psi _\varepsilon]^c = \emptyset.$ Hence $\hat{w}_\varepsilon \leq \psi _\varepsilon,$ whenever $|x| \geq R_\varepsilon,$ and one can choose $C=C_\varepsilon>0$ to conclude $\hat{w}_\varepsilon \leq C \psi _\varepsilon$ in $\mathbb{R}^N.$ That is,
		\begin{equation*}
		u_\varepsilon (x) \leq  C\psi \left(\frac{x-\varepsilon\hat{z}_\varepsilon}{\varepsilon} \right)=C \exp\left(-\alpha \left|\frac{x-z_\varepsilon}{\varepsilon} \right|\right),\quad \forall \, x \in \mathbb{R}^N.\qedhere
		\end{equation*}

(iv) We notice that $\hat{w}_\varepsilon(x)=u_\varepsilon (\varepsilon x + z_\varepsilon)$ is a solution of the following equation, 
\begin{equation}\label{fim}
    -\Delta _p w + V(\varepsilon (x + \hat{z}_\varepsilon) ) w^{p-1} = g(\varepsilon (x + \hat{z}_\varepsilon),w),\ w>0,\ \text{in} \ \mathbb{R}^N,
\end{equation}
as defined in Section \ref{s_uniform}. Taking any $\varepsilon = \varepsilon_n \rightarrow 0,$ by \eqref{norma_est}, the sequence $\| \nabla \hat{w}_{\varepsilon_n}\| _p$ is bounded, and there is $\hat{w} \in W^{1,p}(\mathbb{R}^N)$ such that $ \hat{w}_\varepsilon \rightharpoonup \hat{w}$ in $D^{1,p}(\mathbb{R}^N),$ up to a subsequence. Proposition \ref{p_convergcomp} yields $\hat{w}_{\varepsilon_n} \rightarrow \hat{w}$ in $C^1(\overline{K}),$ for any compact set $K\subset \mathbb{R}^N.$ Since $\varepsilon_n (x + \hat{z}_{\varepsilon_n} ) \rightarrow x_{\Omega},$ by the arguments of Proposition \ref{p_convergcomp}, $\hat{w} \in C^{1,\alpha}_{\loca} (\mathbb{R}^N)$ and it is a solution of the equation \eqref{p_limit}, because $V(x_{\Omega}) = V(0).$ Moreover, since $\hat{w}_\varepsilon \rightarrow \hat{w},$ as $|x| \rightarrow \infty,$ uniformly in $\varepsilon \in (0,\varepsilon_0),$ we have $\hat{w}(x) \rightarrow 0,$ as $|x| \rightarrow \infty .$ It remains to prove that $\hat{w}$ is a ground state of \eqref{PGS}. Denote $\hat{I}_\varepsilon $ the energy functional defined in $X_\varepsilon$ related to \eqref{fim}. By definition, $\hat{w}_\varepsilon(x) = v_\varepsilon (x + \hat{z}_\varepsilon),$ implying $\hat{I}_\varepsilon (\hat{w}_\varepsilon) = I_\varepsilon (v_\varepsilon)=c(I_\varepsilon).$ Thus, Proposition \ref{p_estbaixo} and Fatou's lemma imply
\begin{equation*}
    c(\mathcal{I}_0)=\liminf_{n \rightarrow \infty } \hat{I}_{\varepsilon_n} (\hat{w}_{\varepsilon_n}) \geq \frac{1}{p} \int _{\mathbb{R}^N} g(x_\Omega , \hat{w}) \hat{w} - p G(x_\Omega , \hat{w})\dx = \mathcal{I}_0(\hat{w}).
\end{equation*}
Nevertheless, it is known that $c(\mathcal{I}_0) \leq \max_{t \geq 0} \mathcal{I}_0(t \hat{w}) = \mathcal{I}_0 (\hat{w})$ \cite{wang-an-zhang2015}. Hence $\hat{w}$ is a solution of \eqref{fim} at the mountain pass level $c(\mathcal{I}_0).$ On the other hand, let $\mathcal{N}_0 = \left\{ v \in W^{1,p} (\mathbb{R}^N)\setminus \{  0\} : \mathcal{I}'_0 (v) \cdot v = 0 \right\} $ be the Nehari set associated to $\mathcal{I}_0.$ By \ref{f_increasing} it is also known (see \cite[Lemma 3.2]{wang-an-zhang2015}) that $ c_{\mathcal{N}_0} = c(\mathcal{I}_0) =\mathcal{I}_0 (\hat{w}),$ where $c_{\mathcal{N}_0} = \inf _{v \in \mathcal{N}_0 }\mathcal{I}_0(v).$ In this case, we have 
\begin{equation*}
     \mathcal{I}_0 (\hat{w}) =  c_{\mathcal{N}_0} = \inf \left\{ \mathcal{I}_0(v): v \in W^{1,p} (\mathbb{R}^N)\setminus \{  0\}\text{ and } \mathcal{I}'_0(v) =0 \right\},
\end{equation*}
which means that $\hat{w}$ is a ground state solution of \eqref{PGS}. 

\end{proof}

\section*{Acknowledgements}
The author would like to thank Ailton Rodrigues da Silva (DMAT/UFRN) for several discussions about this subject.

\end{document}